\documentclass[a4paper,10pt,leqno]{amsart}
\usepackage{amsmath,amsfonts,amsthm,amssymb,dsfont}
\usepackage[alphabetic]{amsrefs}
\usepackage{hyperref}
\usepackage[OT4]{fontenc}
\usepackage{enumerate}
\usepackage{mathrsfs}
\usepackage{enumerate, xspace}
\usepackage{graphicx}
\usepackage{color}

\long\def\symbolfootnote[#1]#2{\begingroup
\def\thefootnote{\fnsymbol{footnote}}\footnote[#1]{#2}\endgroup}

\newtheorem{theorem}{Theorem}[section]
\newtheorem{lemma}[theorem]{Lemma}
\newtheorem{lem}[theorem]{Lemma}
\newtheorem{thm}[theorem]{Theorem}

\newtheorem{prop}[theorem]{Proposition}

\newtheorem{cor}[theorem]{Corollary}

\newtheorem{metathm}{Metatheorem}

\theoremstyle{definition}

\newtheorem{rem}[theorem]{Remark}
\newtheorem*{claim}{Claim}
\newtheorem{defin}[theorem]{Definition}

\newtheorem{conj}[theorem]{Conjecture}

\newcommand{\N}{\mathbf{N}}
\newcommand{\R}{\mathbf{R}}
\newcommand{\Z}{\mathbf{Z}}
\renewcommand{\H}{\mathbf{H}}
\newcommand{\C}{\mathbf{C}}

\newcommand {\mr}{\mathrm}

\begin{document}

\title{Realisation and dismantlability}

\author[S.~Hensel]{Sebastian Hensel}
           \address{Mathematisches Institut der Universit\"at Bonn\\
             Endenicher Allee 60, D-53115 Bonn, Germany}
           \email{loplop@math.uni-bonn.de}

\author[D.~Osajda]{Damian Osajda$^\dag$}
           \address{Universit\"at Wien, Fakult\"at f\"ur Mathematik\\
Nordbergstra\ss e 15, 1090 Wien, Austria\\
and}
\address{Instytut Matematyczny, Uniwersytet Wroc{\l}awski\\
pl. Grunwaldzki 2/4, 50-384 Wroc{\l}aw, Poland}
           \email{dosaj@math.uni.wroc.pl}
\thanks{$\dag$ Partially supported by MNiSW grant N N201 541738.}

\author[P.~Przytycki]{Piotr Przytycki$^{\dag\ddag}$}
\address{Inst. of Math., Polish Academy of Sciences\\
 \'Sniadeckich 8, 00-956 Warsaw, Poland}
\email{pprzytyc@mimuw.edu.pl}
\thanks{$\ddag$ Partially supported by the Foundation for Polish Science.}

\maketitle

\begin{abstract}
\noindent
We study dismantling properties of the arc, disc and sphere graphs. We
prove that any finite subgroup $H$ of the mapping class group of a
surface with punctures, the handlebody group, or $\mathrm{Out}(F_n)$
fixes a filling (resp.\ simple) clique in the appropriate graph. We
deduce realisation theorems, in particular the Nielsen Realisation
Problem in the case of a nonempty set of punctures. We also prove that
infinite $H$ have either empty or contractible fixed point sets in the
corresponding complexes. Furthermore, we show that their spines are
classifying spaces for proper actions for mapping class groups and
$\mathrm{Out}(F_n)$.
\end{abstract}

\section{Introduction}
\label{sec:intro}

One of the most outstanding long open problems in mapping class groups is
the \emph{Nielsen Realisation Problem}, which asks if a finite
group of isotopy classes of homeomorphisms of a surface can be
realised as an actual group of homeomorphisms.
This question was answered affirmatively in 1980 by
Steven Kerckhoff, by in fact realising the isotopy classes as
isometries of a suitable hyperbolic metric.

\begin{thm}[\cite{Ker, Ker2}]
\label{thm:Nielsen}
Let $X$ be a connected oriented surface of finite type and negative Euler characteristic.
Let $H$ be a finite subgroup of the mapping class group $\mathrm{Map}(X)$. Then $X$ admits a complete hyperbolic metric such that $H$ acts on $X$ as a group of isometries.
\end{thm}

Kerckhoff's proof is analytic, based on proving convexity of length functions along ``earthquake paths''.
Alternate but also involved proofs were later given by Wolpert \cite{Wol} and Gabai \cite{Gab}.

Similar realisation theorems are true in other contexts as well. As one
example, we can consider a handlebody $U$ instead of a surface. Its mapping class group $\mathrm{Map}(U)$ is called the \emph{handlebody group}.

\begin{thm}[\cite{Zim2}]
\label{thm:schottky-nielsen}
  Let $U$ be a connected handlebody of genus $\geq 2$. Let $H$ be a finite subgroup
  of the handlebody group $\mathrm{Map}(U)$. Then there is a hyperbolic $3$--manifold $M$
  which is homeomorphic to the interior of $U$ such that $H$ acts on
  $M$ as a group of isometries.
\end{thm}

Theorem~\ref{thm:schottky-nielsen} may be deduced from Theorem~\ref{thm:Nielsen} using Marden isomorphism theorem (see \cite{Zim2}).

As a last example of realisation theorems we consider graphs. The group $\mathrm{Out}(F_n)$ acts by homotopy equivalences on a graph with fundamental group $F_n$.

\begin{thm}[\cite{Zim,Cul,Kh}]
\label{thm:graph}
Let $H$ be a finite subgroup of $\mathrm{Out}(F_n)$. Then there is a graph with fundamental group $F_n$ on which $H$ acts as a group of isometries.
\end{thm}

The proofs require the description of virtually free groups as graphs of finite groups.

Theorem~\ref{thm:graph} is known to imply Theorem~\ref{thm:Nielsen} for punctured surfaces. Namely, a finite subgroup $H$ of $\mathrm{Map}(X)$ with $\pi_1(X)=F_n$ embeds in $\mathrm{Out}(F_n)$. By
Theorem~\ref{thm:graph} the group $H$ acts by isometries
on a graph $\Gamma$ with $\pi_1(\Gamma)=F_n$. One can then show that $\Gamma$ has an $H$--invariant ribbon
graph structure determining the action of $H$ on $X$. However, to better illustrate our approach,
we will treat Theorem~\ref{thm:Nielsen} separately.   

\medskip
The purpose of this article is to develop a unified and elementary combinatorial
approach to such realisation problems. We give proofs of
Theorem~\ref{thm:schottky-nielsen} and Theorem~\ref{thm:graph} in full
generality, and
we also prove Theorem~\ref{thm:Nielsen} under any of the following hypotheses:
\begin{enumerate}[(A)]
\item
The set of punctures of $X$ is nonempty.
\item
There is a nonempty set of disjoint essential simple closed curves on $X$, which is $H$--invariant up to homotopy.
\end{enumerate}

Our method is elementary and consists in finding the fixed point of the action of $H$ on the arc graph (resp. the disc graph, the sphere graph).

\begin{metathm}
\label{met:filling}
Let $\Gamma$ be the arc graph (resp.\ the disc graph, the sphere graph). Then any finite subgroup of the mapping class group (resp.\ the handlebody group, $\mathrm{Out}(F_n)$) fixes
\begin{enumerate}[(1)]
\item
a clique in $\Gamma$,
\item
a filling (resp.\ simple) clique in $\Gamma$.
\end{enumerate}
\end{metathm}

Assertion (2) of Metatheorem~\ref{met:filling} is stronger than (1), but in order to prove (2), we first prove (1) and then apply induction.
To prove (1), the first step is to use an easy surgery procedure to show that the
graph $\Gamma$ can be exhausted by finite subgraphs enjoying a combinatorial nonpositive curvature
feature called \emph{dismantlability}. We then show that, even though $\Gamma$ is locally infinite, a finite
subgroup of the ambient group preserves a finite subgraph that is dismantlable. By a theorem of Polat \cite{Pol} a finite dismantlable graph contains a clique fixed by all of its automorphisms.

\medskip
Dismantlability is a feature shared by many graphs
associated to mapping class groups and related groups.
In this article, we show it for the arc graph of a punctured surface,
the disc graph of a handlebody and the sphere graph of a doubled handlebody.
We note that the proofs of dismantlability follow easily from the
known proofs of contractibility for the arc and disc graph, but not
for the sphere graph. We also prove dismantlability for the Rips graph of a hyperbolic group.

Additionally, dismantlability is known for Kakimizu graphs of
link complement groups and $1$--skeleta of weakly systolic
complexes. In \cite{PS} we used the Kakimizu graph to prove a fixed
point theorem for finite subgroups of link complements. We proved a
weakly systolic fixed point theorem in \cite{CO}.

\medskip

In the second part of the article, we generalise Polat's theorem to the following.

\begin{thm}
\label{thm:contrfix}
  Let $\Gamma$ be a finite dismantlable graph and let $Y=\Gamma^\blacktriangle$, the flag simplicial complex spanned on $\Gamma$.
  Let $H$ be a finite subgroup of $\mr{Aut}(Y)$. Then the set $Y^H$ of points of $Y$ fixed by $H$ is contractible.
\end{thm}

Theorem~\ref{thm:contrfix} leads to the following metatheorems.

\begin{metathm}
\label{met:H_infinite}
Let $H$ be any subgroup of the mapping class group (resp.\ the handlebody group, $\mathrm{Out}(F_n)$). The fixed point set of the action of $H$ on the arc complex (resp.\ the disc complex, the sphere complex) is empty or contractible.
\end{metathm}

In the case where $H$ is the trivial group Metatheorem~\ref{met:H_infinite} implies that the arc, disc and
sphere complexes are contractible. These facts were known before, see
\cite{Har1,Ha0,McC,Ha}.
For general $H$ the results in Metatheorem~\ref{met:H_infinite} are
new. In particular for $H$ finite, in view of
Metatheorem~\ref{met:filling}(1), the fixed point set of $H$ is
contractible.
\medskip

The following builds on and generalises Metatheorem~\ref{met:filling}(2).

\begin{metathm}
\label{met:H_filling}
Let $H$ be a finite subgroup of the mapping class group (resp.\ the handlebody group, $\mathrm{Out}(F_n)$). The fixed point set of the action of $H$ on the filling arc system complex (resp.\ the simple disc system complex, the simple sphere system complex) is contractible.
\end{metathm}

Consequently the filling arc system complex, as well as the simple
sphere system complex, is a finite model for $\underline{E}G$ --- the
classifying space for proper actions for the ambient group $G$. This
was known for $\mathrm{Out}(F_n)$ (\cite{Whi}, \cite{KV}) for which the
corresponding complex is the spine of the Outer space \cite{Ha}. For
the mapping class group and the handlebody group
Metatheorem~\ref{met:H_filling} is new. When $G$ is the mapping class
group other finite models for $\underline{E}G$ were constructed before
in \cite{JW,Mis}.
But it was not known if the filling arc system complex, which is a
natural candidate, is in fact a model for $\underline{E}G$. In the
case of a surface with one
puncture the filling arc system complex coincides with Harer's spine
of the Teichm\"uller space \cite{Har2}. Our method also answers
affirmatively the question if Harer's spine is in general a model for
$\underline{E}G$ for $G$ the pure mapping class group.

Theorem \ref{thm:contrfix} can be applied in many other situations. In particular it can be used to show contractibility of sets of fixed points for weakly systolic complexes (cf.\ \cite{CO}) or for Kakimizu complexes (cf.\ \cite{PS}). In our article we show how to deduce from Theorem~\ref{thm:contrfix} a new proof of the following classical result.

\begin{thm}[{\cite[Thm~1]{MS}}]
\label{thm:EGhyper}
Let $G$ be a $\delta$--hyperbolic group and let $D\geq 8\delta +1$. Then the Rips complex $P_D(G)$ is a finite model for $\underline EG$.
\end{thm}

We believe that many simplicial complexes appearing elsewhere in the literature can be studied using dismantlability and we intend to continue developing this approach.

\subsection*{Organisation}
The article is divided into two parts. In Sections~\ref{sec:dism}--\ref{sec:sph-red} we prove Metatheorem~\ref{met:filling} and deduce realisation results. In Sections~\ref{sec:contr}--\ref{sec:Rips} we prove Metatheorems~\ref{met:H_infinite} and~\ref{met:H_filling}. Each part is preluded by a graph-theoretic Section \ref{sec:dism} or~\ref{sec:contr}.

In Section~\ref{sec:dism} we discuss the notion of dismantlability. We include the proof of Polat's theorem (Theorem~\ref{thm:fixedpoint} in the text) and deduce a fixed clique criterion which will be used for Metatheorem~\ref{met:filling}(1).
In Section~\ref{sec:fixarc} we prove Metatheorem~\ref{met:filling}(1) for the mapping class group (Theorem~\ref{thm:fixed_arc}). In Section~\ref{sec:Nielsen} we deduce from this Metatheorem~\ref{met:filling}(2) for the mapping class group (Theorem~\ref{thm:fixed_arc_system}) and Theorem~\ref{thm:Nielsen} under hypothesis (A) or (B).

In Section~\ref{sec:fixdisc} we prove Metatheorem~\ref{met:filling}(1) for the handlebody group (Theorem~\ref{thm:fixed_disc}). In
Section~\ref{sec:Nielsen_for_handle} we derive Metatheorem~\ref{met:filling}(2) for the handlebody group (Theorem~\ref{thm:fixed_disc_system}) and Theorem~\ref{thm:schottky-nielsen}. Similarly, we prove Metatheorem~\ref{met:filling}(1) for $\mathrm{Out}(F_n)$ (Theorem~\ref{thm:fixedsphere}) in Section~\ref{sec:spheres} and promote it to Metatheorem~\ref{met:filling}(2) for $\mathrm{Out}(F_n)$ (Theorem~\ref{thm:fixed-filling-sphere-clique}) and to Theorem~\ref{thm:graph} in Section~\ref{sec:sph-red}.

The second part of the article begins with Section~\ref{sec:contr}, where we prove Theorem~\ref{thm:contrfix}. This section can be read alternatively directly as a continuation of Section~\ref{sec:dism}. In Section~\ref{sec:contr_arc} we deduce from Theorem~\ref{thm:contrfix} the results of Metatheorems~\ref{met:H_infinite} and~\ref{met:H_filling} for the mapping class group (Theorems~\ref{thm:contrarc} and~\ref{thm:arc_classify}). We do the same for the handlebody group and $\mathrm{Out}(F_n)$ in Section~\ref{sec:contr_disc_sphere} (Theorems~\ref{thm:contrdisc}, \ref{thm:handle_spine}, \ref{thm:contrsphere}, and~\ref{thm:sphere_spine}). Consequently, we obtain finite models for classifying spaces for proper actions in Corollaries~\ref{cor:spine_teich}, \ref{cor:spine_teich_Harer}, and~\ref{cor:spine_sphere}. We close the article with the proof of Theorem~\ref{thm:EGhyper} in Section~\ref{sec:Rips}.

\subsection*{Acknowledgements}
We would like to thank Victor Chepoi, Saul Schleimer and Karen Vogtmann for discussions. We
thank the Mittag--Leffler Institute, where a large part of this work
was carried out, for its hospitality and a great working atmosphere.

\section{Dismantlability}
\label{sec:dism}

Our goal in this section is to give a criterion for a group of
automorphisms of a graph to have an invariant clique
(Proposition~\ref{prop:fixedpoint}). The criterion requires the
existence of ``projections'' and finite sets which are ``convex''
under sufficiently many of these projections. This criterion will be used to obtain Metatheorem~\ref{met:filling}(1), where the role of projections will be played by surgeries.

Our proof of Proposition~\ref{prop:fixedpoint} relies on Polat's fixed point theorem for dismantlable graphs \cite{Pol}. We include the proof of a special case of Polat's theorem for the article to be self-contained. The notion of dismantlability was brought into geometric group theory by Chepoi and the second named author, and it was used in particular to prove a fixed point theorem for weakly systolic complexes \cite{CO}. Later Schultens and the third named author used it to obtain contractibility and a fixed point theorem for the Kakimizu complex \cite{PS}.

The results of this section have a natural continuation and generalisation in Section~\ref{sec:contr}.

Unless stated otherwise, all graphs in this article are \emph{simple} in the sense that they do not have double edges or loops, i.e.\ edges connecting a vertex to itself.
Let $\Gamma$ be a graph with vertex set $V$. A subgraph of $\Gamma$ \emph{induced} on a subset $V'\subset V$ has vertex set $V'$ and all of the edges of $\Gamma$ connecting the vertices in $V'$.

\begin{defin}
For a vertex $\rho$ of a graph its \emph{neighbourhood} $N(\rho)$ is the set of vertices consisting of $\rho$ and all its neighbours. We say that a vertex $\rho$ of a graph is \emph{dominated} by a (\emph{dominating}) vertex $\pi \neq \rho$, if $N(\rho)\subset N(\pi)$. The symbol $\subset$ here allows equality.

A finite graph is \emph{dismantlable} if its vertices can be ordered into a sequence $\sigma_1, \ldots , \sigma_m$ so that for each $i<m$ the vertex $\sigma_i$ is dominated in the subgraph induced on $\{\sigma_i,\ldots,\sigma_m\}$. We call such an order a \emph{dismantling order}.
\end{defin}

\begin{defin}
A \emph{clique} in a graph is a nonempty subgraph all of whose vertices are pairwise connected by edges. The vertex set of a clique is called a \emph{clique} as well.

For a graph $\Gamma$ let $\Gamma^\blacktriangle$ denote the simplicial complex with $1$--skeleton $\Gamma$ and simplices spanned on all cliques in $\Gamma$.
A simplicial complex $Y$ is \emph{flag} if $Y=\Gamma^\blacktriangle$, where $\Gamma$ is the $1$--skeleton of $Y$.
If $S$ is a subset of the vertex set of $\Gamma$, then $S^\blacktriangle$ denotes the subcomplex of  $\Gamma^\blacktriangle$ spanned on $S$.
\end{defin}

\begin{rem}
If a finite graph $\Gamma$ is dismantlable, then $\Gamma^\blacktriangle$ is contractible. This will be generalised in Section~\ref{sec:contr}, where we prove contractibility of fixed point sets of group actions.
\end{rem}

We build on the following special case of a result of Polat (who proves an analogue for some infinite graphs).

 \begin{thm}[{\cite[Thm~A]{Pol}}]
\label{thm:fixedpoint}
A finite dismantlable graph contains a clique invariant under all of its automorphisms.
\end{thm}
Note that there are finite group actions on finite contractible or
even collapsible simplicial complexes with no fixed
points. Dismantlability seems to be the correct strengthening of those notions.

For completeness we include a concise proof of
Theorem~\ref{thm:fixedpoint}. Roughly speaking, the invariant clique
is obtained by successively removing simultaneously all dominated
vertices.

\begin{lem}
\label{lem:pushing}
If a finite graph with vertex set $V$ is dismantlable and $\sigma\in V$ is dominated, then the subgraph induced on $V\setminus \{\sigma\}$ is dismantlable as well.
\end{lem}
\begin{proof}
By dismantlability we can order the vertices of $V$ into a sequence $\sigma_1, \ldots, \sigma_m$ so that
for each $i<m$ the vertex $\sigma_i$ is dominated in the subgraph induced on $V^i=\{\sigma_i,\ldots, \sigma_m\}$ by some $\sigma_{f(i)}$ with $f(i)>i$. Assume that $\sigma=\sigma_j$ is dominated in $V$ by $\sigma_k$. Note that we might have $k<j$.

Let $k_1=k$. For $i=2,\ldots, j$ define $k_i=k_{i-1}$ if $k_{i-1}\neq i-1$ and $k_i=f(k_{i-1})$ otherwise. Observe that $k_i\geq i$ and $N(\sigma_{k_i})\cap V^i\supset N(\sigma_j)\cap V^i$.

First assume that there is no $i$ with $k_i=j$. In that case we induce an order on $V\setminus \{\sigma\}$ from $V$ and put $f'(i)=f(i)$ if $f(i)\neq j$ and $f'(i)=k_i$ otherwise. Let $i<m$ be distinct from $j$. Then $\sigma_i$ is dominated in the subgraph induced on $V^i\setminus \{\sigma\}$ by $\sigma_{f'(i)}$. Hence this is a dismantling order.

If $l$ is minimal with $k_{l+1}=j$, then $k_l=l$ and we have $N(\sigma_{l})\cap V^l= N(\sigma_j)\cap V^l$. Hence the subgraph induced on $V^{l}\setminus \{\sigma\}$ is isomorphic to the subgraph induced on $V^{l+1}$ which was dismantlable. Then in the order on $V\setminus \{\sigma\}$ induced from $V$ we carry $\sigma_l$ to the position $j$ which was occupied in $V$ by $\sigma$. For every $i<l$ the vertex $\sigma_i$ is dominated in the subgraph induced on $V^i\setminus \{\sigma\}$ by $\sigma_{f'(i)}$ as before.
\end{proof}

\begin{proof}[Proof of Theorem~\ref{thm:fixedpoint}.]
The proof is by induction on the number of vertices. If the graph consists of one vertex, the theorem is trivial. Now assume we have a dismantlable graph $\Gamma$ with vertex set $V$ and the theorem is true for all graphs with smaller number of vertices. We treat separately two cases.

\medskip
\noindent \emph{Case 1.} There are no two vertices $\sigma, \rho\in V$ with $N(\sigma)=N(\rho)$. In this case, let $V'\subset V$ be the set of all dominated vertices. Note that each vertex of $V'$ is dominated by a vertex in $V\setminus V'$. Let $\Gamma'$ be the subgraph of $\Gamma$ induced on $V\setminus V'$.
Note that every automorphism of $\Gamma$ restricts to an automorphism
of  $\Gamma'$. Since $\Gamma'$ is obtained from $\Gamma$ by repeatedly
removing dominated vertices, by Lemma~\ref{lem:pushing} the graph
$\Gamma'$ is dismantlable. By induction hypothesis, there is a clique
$\Delta'$ in $\Gamma'$ invariant under every automorphism of
$\Gamma'$. The clique $\Delta'$ is then also invariant under every
automorphism of $\Gamma$.

\medskip
\noindent \emph{Case 2.} There are vertices $\sigma, \rho\in V$ with $N(\sigma)=N(\rho)$. In this case, let $\Gamma'$ be the graph obtained from $\Gamma$ by identifying all vertices $\sigma$ with common $N(\sigma)$ (and identifying the double edges and removing loops). More precisely, we consider the equivalence relation $\sim$ on $V$ for which $\sigma\sim \rho$ if and only if $N(\sigma)=N(\rho)$. Equivalence classes of $\sim$ are cliques. The vertex set of $\Gamma'$ is the set of equivalence classes of $\sim$. Two vertices of $\Gamma'$ are connected by an edge if some (hence any) of its representatives in $\Gamma$ are connected by an edge.

Note that every automorphism of $\Gamma$ projects to an automorphism
of $\Gamma'$. Moreover,  $\Gamma'$ can be embedded as an induced
subgraph in $\Gamma$ and obtained from $\Gamma$ by repeatedly removing
dominated vertices. Again, by Lemma~\ref{lem:pushing} the graph
$\Gamma'$ is dismantlable and by induction hypothesis there is a
clique $\Delta'$ in $\Gamma'$ invariant under every automorphism of
$\Gamma'$. Hence its preimage $\Delta$ under the projection from
$\Gamma$ is a clique invariant under every automorphism of $\Gamma$.
\end{proof}

Dismantlability can be verified if we have particular ``projection'' maps.

\begin{defin}
\label{def:proj}
Consider a graph with vertex set $V$. Let $\sigma\in V$ and let $\Pi$ assign to each vertex $\rho\in V\setminus \{\sigma\}$ a nonempty finite set $\Pi(\rho)$ consisting of pairs of elements of $V$.
We allow the two elements of a pair to coincide. Let $\Pi^*(\rho)\subset V$ denote the set of all vertices appearing in pairs from $\Pi(\rho)$.

We call $\Pi$ a \emph{$\sigma$--projection} (or a \emph{dismantling projection} if $\sigma$ is not specified) if the following axioms are satisfied:
\begin{enumerate}[(i)]
\item
For each finite set of vertices $R\subset V$ with nonempty $R\setminus\{\sigma\}$ there is an \emph{exposed} vertex $\rho\in R\setminus \{\sigma\}$, that is a vertex with $N(\rho)\cap R\subset N(\pi)$ for both $\pi$ from some pair of $\Pi(\rho)$.
\item
There is no cycle of vertices $\rho_0,\ldots, \rho_{m-1}\in V$ with $\rho_{i+1}\in \Pi^*(\rho_i)$, where $i$ is considered modulo $m$.
\end{enumerate}
\end{defin}

Note that axiom (ii) implies in particular that $\rho\notin \Pi^*(\rho)$.

\begin{lem}
\label{lem:proj->dism}
A finite graph with a dismantling projection is dismantlable.
\end{lem}

Before we provide the proof of Lemma~\ref{lem:proj->dism} we deduce a corollary.

\begin{defin}
\label{def:convex}
Consider a graph with vertex set $V$ and a $\sigma$--projection $\Pi$. We say that a subset $R\subset V$ is \emph{$\Pi$--convex} if for every vertex $\rho\in R\setminus \{\sigma\}$ each pair in $\Pi(\rho)$ intersects $R$.
\end{defin}

We will see in a moment that if $R$ is $\Pi$--convex for a $\sigma$--projection $\Pi$, then $\sigma \in R$.

\begin{cor}
\label{cor:proj->dism}
Let $\Gamma$ be a graph with a dismantling projection $\Pi$ and a finite $\Pi$--convex subset $R$ of its vertices. Then the subgraph of $\Gamma$ induced on $R$ is dismantlable.
\end{cor}

\begin{proof}
Suppose $\Pi$ is a $\sigma$--projection. For any pair $P$ intersecting $R$ consider $P\cap R$ as a pair (of possibly two coinciding elements).
Let $\Pi^R$ assign to any $\rho\in R\setminus \{\sigma\}$ the set of pairs $P\cap R$ over $P\in \Pi(\rho)$.
The assignment $\Pi^R$ satisfies axioms (i) and (ii) of a $\sigma$--projection for the induced subgraph. The only non-immediate part is to verify $\sigma\in R$, which can be deduced from axiom (ii): we consider the longest sequence $\rho_0,\ldots, \rho_{m}\in R$ with $\rho_{i+1}\in \Pi^*(\rho_i)$. Since the only vertex on which $\Pi^R$ might not be defined is $\sigma$, we have $\rho_m=\sigma$.
\end{proof}

\begin{proof}[Proof of Lemma~\ref{lem:proj->dism}]
Let $m$ denote the cardinality of the vertex set $V$ of the graph. Suppose the dismantling projection $\Pi$ is a $\sigma$--projection. Define inductively $\sigma_i$ for $i=1,\ldots, m-1$ so that $\sigma_i$ is exposed (see axiom (i) of a $\sigma$--projection) in the set $V^i=V\setminus \{\sigma_1,\ldots, \sigma_{i-1}\}$ and put $\sigma_m=\sigma$. We will verify that this is a dismantling order. Since $\sigma_i$ are exposed, there are vertices $\pi^1_1,\ldots, \pi^1_{m-1}$ in $\Pi^*(\sigma_1),\ldots,\Pi^*(\sigma_{m-1})$ satisfying $N(\sigma_i)\cap V^i\subset N(\pi^1_i)$.

By construction the vertex $\sigma_1$ is dominated in $V^1=V$ by
$\pi_1^1$. However, the vertices $\sigma_i$ for $i>1$ might not be dominated in $V^i$ by $\pi_i^1$ if the latter lie outside $V^i$. This might happen if $\sigma_j$ for some $j<i$ is chosen to be $\pi_i^1$. However, in that case we can replace $\pi_i^1$ with $\pi^1_j$, if the latter is in $V^i$, or else with yet another vertex. Here is the systematic way to determine this choice.

By axiom (ii) of a $\sigma$--projection, the directed graph with vertices $V=V^1$ and edges $\{(\sigma_j, \pi^1_j)\}_{j=1}^{m-1}$ has no directed cycles. For $i=2$ to $m-1$ inductively define $\pi^i_j$, where $j=i,i+1,\ldots,m-1$ so that $\pi^i_j=\pi^{i-1}_j$ if $\pi^{i-1}_j\neq\sigma_{i-1}$, and $\pi^i_j=\pi^{i-1}_{i-1}$ otherwise. We inductively see that in both cases $\pi^i_j\in V^i$ and the directed graph with vertices $V^i$ and edges $\{(\sigma_j, \pi^i_j)\}_{j=i}^{m-1}$ has no directed cycles. We also claim the following, which we will prove by induction on $i=1,\ldots, m-1$:

\begin{enumerate}[(a)]
\item
$N(\pi_i^1)\cap V^i\subset N(\pi^i_i)$,
\item
$N(\sigma_i)\cap V^i\subset N(\pi^i_i)$,
\item
$N(\pi_j^i)\cap V^{i+1}\subset N(\pi_j^{i+1})$  for $j=i+1,\ldots, m-1$.
\end{enumerate}

Note that part (b) means that $\sigma_i$ is dominated in $V^i$, which implies dismantlability.

First observe that part (a) implies part (b) by the choice of $\pi_i^1$. Also note that part (b) implies directly part (c) since in the case where $\pi_j^{i+1}=\pi_j^i$ there is nothing to prove and otherwise $\pi_j^i=\sigma_i$ and $\pi_j^{i+1}=\pi^i_i$. Finally, part (a) for $i=k$ follows from applying $k-1$ times part (c) with $j=k$ and $i=1,\ldots, k-1$.
\end{proof}

We are now in position to prove our main criterion for the existence of an invariant clique.

\begin{defin}
Let $H$ be a finite group of automorphisms of a graph with vertex set $V$. A family of $\sigma$--projections $\{\Pi_\sigma\}_{\sigma\in S}$ with $S\subset V$ is \emph{$H$--equivariant} if $S$ is $H$--invariant and
$$h\Pi_{\sigma}(\rho)=\Pi_{h\sigma}(h\rho)$$
for all $\sigma\in S, \ \rho\in V$ and $h\in H$.
\end{defin}

\begin{prop}
\label{prop:fixedpoint}
Let $H$ be a group of automorphisms of a graph $\Gamma$ with vertex set $V$. Suppose that there is
a $\sigma$--projection $\Pi_{\sigma}$
and a finite $\Pi_{\sigma}$--convex subset $R\subset V$. Furthermore suppose that

\begin{enumerate}[(i)]
\item
$R$ is $H$--invariant, or
\item
the set $HR$ is finite and for all $\sigma'\in S=H\sigma\subset V$ there are $\sigma'$--projections $\Pi_{\sigma'}$ for which $R$ is $\Pi_{\sigma'}$--convex. Moreover, the family $\{\Pi_{\sigma'}\}_{\sigma'\in S}$ is $H$--equivariant.
\end{enumerate}
Then the subgraph induced on $HR\subset V$ is dismantlable and $\Gamma$ contains an $H$--invariant clique.
\end{prop}

\begin{proof}
Part (i) follows directly from Corollary~\ref{cor:proj->dism} and Theorem~\ref{thm:fixedpoint}.

In part (ii), let $T=HR\subset V$. Then the subgraph induced on the finite vertex set $T$ is $H$--invariant and we want to reduce to part (i) with $T$ in place of $R$. Let $\tau\in T\setminus\{\sigma\}$. Then $\tau=h\rho$ for some $\rho\in R, \ h\in H$. Let $P$ be a pair in $\Pi_{\sigma}(\tau)=\Pi_{hh^{-1}\sigma}(h\rho)$. Then $h^{-1}(P)$ is a pair in $\Pi_{h^{-1}\sigma}(\rho)$. Since $h^{-1}\sigma\in S$ the set $R$ is $\Pi_{h^{-1}\sigma}$--convex. Then $h^{-1}(P)\cap R$ is nonempty and hence $h^{-1}(P)\cap T$ is nonempty. Since $T=h^{-1}T$ we have $P\cap T\neq \emptyset$, as desired.
\end{proof}

We conclude this section by stating a conjecture, whose validity
would significantly simplify the proofs of our realisation results and
of similar results in future. Define an infinite graph $\Gamma$ to be \emph{dismantlable} if the following holds: every finite set of vertices of $\Gamma$ is contained in another finite set of vertices $S$ such that the finite subgraph induced on $S$ is dismantlable.

\begin{conj}
Let $H$ be a finite group of automorphisms of a dismantlable graph $\Gamma$. Then $\Gamma$ contains an $H$--invariant clique.
\end{conj}

\section{Fixed clique in the arc graph}
\label{sec:fixarc}
In this Section we prove Metatheorem A(1) in the case of the arc graph.
Let $X$ be a (possibly disconnected) closed oriented surface with
a nonempty finite set $\mu$ of \emph{marked points} on $X$.
We require that homeomorphisms and homotopies of $X$ fix $\mu$.
The \emph{Euler characteristic} of $X$ is defined to be the Euler
characteristic of $X - \mu$.

A \emph{simple arc} on $X$ is an embedding of an open interval in $X-\mu$ which can be extended to map the endpoints into $\mu$.
We require that homotopies of arcs fix the endpoints and do not
pass over marked points.
An arc is \emph{essential}, if it is not homotopic into $\mu$.
Unless stated otherwise, all arcs are assumed to be simple and essential.

The \emph{arc graph} $\mathcal{A}(X)$ is the graph
whose vertex set $\mathcal{A}^0(X)$ is the set of homotopy classes of
arcs. Two vertices are connected by an edge in $\mathcal{A}(X)$ if the
corresponding arcs can be realised disjointly.

By $\mathrm{Map}(X)$ we denote the \emph{mapping class group} of $X$,
that is the
group of orientation preserving homeomorphisms of $X$ up to isotopy. The action of
$\mathrm{Map}(X)$ on homotopy classes of arcs induces an action of
$\mathrm{Map}(X)$ on $\mathcal{A}(X)$ as a group of automorphisms.
We can now state the main theorem of this section.
\begin{thm}
\label{thm:fixed_arc}
Let $X$ be a closed oriented surface with nonempty
set of marked points, each component of which
has negative Euler characteristic.
Let $H$ be a finite subgroup of $\mathrm{Map}(X)$.
Then $H$ fixes a clique in the arc graph $\mathcal{A}(X)$.
\end{thm}
We will obtain Theorem~\ref{thm:fixed_arc} using
Proposition~\ref{prop:fixedpoint}(ii). For that we need two elements: on the one hand, we will use
a surgery procedure for arcs to define dismantling projections $\Pi_\sigma$ (see Subsection~\ref{subs:arc_surgery}).
On the other hand, we will define a finite set in the arc graph which is
$\Pi_\sigma$--convex for all $\sigma$ in some $H$--orbit (see Subsection~\ref{subs:hull_arcs}).

\subsection{Arc surgery}
\label{subs:arc_surgery}
We begin by describing the surgery procedure that will be used to
define the dismantling projections. This surgery procedure was used by
Hatcher \cite{Ha0} to show contractibility of the arc complex.

Let $a$ and $a'$ be two arcs on $X$. We say that $a$ and
$a'$ are in \emph{minimal position} if
the number of intersections between $a$ and $a'$ is minimal in the
homotopy classes of $a$ and $a'$.

Now assume that $a$ and $a'$ are in minimal position and suppose that
$a$ and $a'$ intersect.
We say that a subarc $b \subset a$ is \emph{outermost in $a$ for
 $a'$}, if $b$ is a component of $a-a'$ sharing an endpoint with $a$. There are two outermost subarcs in $a$ for $a'$.

Consider an outermost subarc $b$ in $a$ for $a'$. Let $p$ be the endpoint of $b$ in the interior of $a$. Then $p\in a'$.
Let $b'_+, b'_-$ be the two components of $a'-p$.
We say that the arcs $a'_+=b\cup b'_+$ and $a'_-=b\cup b'_-$ are
obtained by \emph{outermost surgery of $a'$ in
direction of $a$ determined by $b$} (see Figure~\ref{fig:arcsurgery}).
Both $a'_+$ and $a'_-$ are essential, since $a$
and $a'$ are in minimal position. Furthermore both $a'_+$ and $a'_-$
are disjoint from $a'$ up to homotopy.

Note that the homotopy classes $\alpha'_+$ of $a'_+$ and $\alpha'_-$
of $a'_-$ depend only on the homotopy classes $\alpha$ of $a$
and $\alpha'$ of $a'$.
We call the pair $\{\alpha'_+,\alpha'_-\}$ an
\emph{outermost surgery pair of $\alpha'$ in direction of $\alpha$}.

\begin{figure}[h!]
\centering
\includegraphics[width=0.35\textwidth]{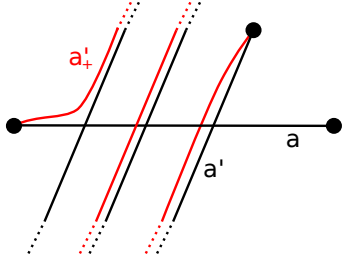}
\caption{Surgery for arcs}
\label{fig:arcsurgery}
\end{figure}

If $\alpha'$ is disjoint from $\alpha$ (but distinct), then we say that the only \emph{outermost surgery pair of $\alpha'$ in direction of $\alpha$} is
$\{\alpha,\alpha\}$, interpreted as a pair both of whose elements are equal to $\alpha$.
Let $\Pi_\alpha(\alpha')$ be the set of all (two or one) outermost surgery pairs of $\alpha'$ in direction of $\alpha$.
If $\alpha' = \alpha$, then $\Pi_\alpha(\alpha')$ is undefined.

\medskip
The following lemma states that the assignment $\Pi_\sigma$ satisfies the
axioms of a $\sigma$--projection for each $\sigma \in \mathcal{A}^0(X)$.
This lemma was essentially proved as Claim~3.18 of \cite{Sch}.
\begin{lem}\label{lem:arcprojection}
  Let $\sigma\in\mathcal{A}^0(X)$ be an arc on $X$.
  \begin{enumerate}[(i)]
  \item Let $R \subset \mathcal{A}^0(X)$ be a finite set of arcs with $R
    \setminus\{\sigma\}\neq\emptyset$. Then
    there is an arc $\rho \in R\setminus\{\sigma\}$ with the
    following property: there is an outermost surgery pair
    $\{\rho_+,\rho_-\}$ of $\rho$
    in direction of $\sigma$, such that each arc $\rho'\in R$ which is
    disjoint from $\rho$ is also disjoint from $\rho_+$ and $\rho_-$.
  \item Every sequence $(\rho_i)$ of arcs in $\mathcal{A}^0(X)$ such that $\rho_{i+1}$ is
    contained in an outermost surgery pair of $\rho_i$ in direction of $\sigma$
    terminates after finitely many steps with the arc $\sigma$.
  \end{enumerate}
\end{lem}
The proof requires us to consistently choose preferred representatives
for each homotopy class of arcs. This is made possible
through the use of hyperbolic geometry. To this end, it is convenient to adopt a
slightly different description of the arc graph of $X$. Namely, let
$\check{X}$ be
the compact surface obtained from $X$ by replacing each marked point
with a boundary component.
Let $\mathcal{A}(\check{X})$ be the graph whose vertex set is the set of
homotopy classes of (essential simple) arcs on $\check{X}$ that are embedded \emph{properly}, i.e.\ they intersect the boundary exactly at the endpoints.
Homotopies of arcs on $\check{X}$ are required to be homotopies of
properly embedded arcs.
Two vertices are connected by an edge if the corresponding arcs may be realised
disjointly.
Since homotopy classes of (properly embedded) arcs on $\check{X}$ are in
one-to-one
correspondence with homotopy classes of arcs on $X$ respecting disjointness, the graph
$\mathcal{A}(\check{X})$ is isomorphic to $\mathcal{A}(X)$. We will
implicitly use this identification, and simply speak about
representatives of arcs $\rho\in\mathcal{A}^0(X)$ on $\check{X}$.

We now fix a hyperbolic metric on $\check{X}$ which makes the boundary
geodesic. Then each homotopy class of an
arc contains a unique shortest geodesic representative, and any two
such representatives are in minimal position (see e.g. the
discussion in \cite[Sec~1.2.7]{FM11}).
\begin{proof}[Proof of Lemma~\ref{lem:arcprojection}]

  \begin{enumerate}[(i)]
  \item
   We can assume that not all the arcs in $R$ are disjoint from $\sigma$, since otherwise the assertion
    follows trivially.
  Denote the shortest geodesic representatives on $\check{X}$ of the elements of $R$ by
  $r_1,\ldots,r_k$ and the shortest geodesic representative of $\sigma$ by $s$.
  Let $b\subset s$ be a component of $s-(r_1\cup\dots\cup r_k)$ sharing an endpoint with $s$. Denote by $p$ the endpoint of $b$ in the interior of $s$. Let $r$ be an arc in $R$
    containing $p$. Suppose $r_*$ is obtained by
    outermost surgery of $r$ in direction of $s$ determined by $b$.
  Let now $r_i$ be disjoint from $r$. Since $r_i$ is disjoint from
    $b$, it is also disjoint from $r_*\subset r\cup b$, as desired.
  \item By construction, each outermost surgery strictly decreases the
    geometric intersection number with $\sigma$. Hence every sequence
    of surgeries eventually yields an arc which is disjoint from
    $\sigma$.
  \end{enumerate}
\end{proof}

For future reference, we also note the following easy fact.
\begin{lem}\label{lem:arcsurgery-equivariance}
  For every $h\in\mathrm{Map}(X)$ we have
  $$h\Pi_\sigma(\rho) = \Pi_{h\sigma}\left(h\rho\right)$$
  for all arcs $\sigma\neq\rho \in \mathcal{A}^0(X)$.
\end{lem}
\begin{proof}
  Let $s, r$ be representatives of $\sigma$ and $\rho$ which are in minimal position. Let $\hbar:X\to X$ be a representative of the mapping class $h$.
  Then $\hbar(s)$ and $\hbar(r)$ are in minimal position, and $\hbar$ maps an
  outermost subarc in $s$ for $r$ to an outermost
  subarc in $\hbar(s)$ for $\hbar(r)$.
\end{proof}

\subsection{Finite hull for arcs}
\label{subs:hull_arcs}
We now describe how a finite vertex set
$\alpha_1,\ldots,\alpha_k\in\mathcal{A}^0(X)$ can be
extended to a finite set which is $\Pi_\sigma$--convex for all $\sigma
= \alpha_i$.

\medskip
Let $\alpha_1,\ldots,\alpha_k\in\mathcal{A}^0(X)$ be an arbitrary finite set
of arcs, and let $a_1,\ldots,a_k$ be their shortest geodesic
representatives on $\check{X}$. Let $L$ be the maximum of the lengths of the $a_i$.

We define $\mathcal{H}=\mathcal{H}(\alpha_1,\ldots,\alpha_k)$ to be the set of
homotopy classes of arcs which have a representative of length at most
$2L$ on $\check{X}$.
The set $\mathcal{H}$ is finite, and contains
$\alpha_1,\ldots,\alpha_k$ by construction.
However, note that the set $\mathcal{H}$ depends on the choice of the
hyperbolic metric.

\begin{lemma}\label{lem:arc-length-reduction}
  Let $\alpha_1,\ldots,\alpha_k\in\mathcal{A}^0(X)$ be a finite set of
  arcs. Put $\alpha = \alpha_i$ for some $i=1,\ldots,k$, and let
  $\alpha' \in \mathcal{H} = \mathcal{H}(\alpha_1,\ldots,\alpha_k)$ be
  arbitrary.
  Then each outermost surgery pair of $\alpha'$ in
  direction of $\alpha$ contains at least one element of $\mathcal{H}$.
\end{lemma}
\begin{proof}
  Let $a'$ be the shortest geodesic representative of $\alpha'$
  and let $a$ be the shortest geodesic representative of $\alpha=\alpha_i$.
  We can assume that $a$ and $a'$ intersect.

  Then an outermost subarc $b$ in $a$ for $a'$ has length at most $L$ by the
  definition of $L$. Let $b'_+,b'_-$ be the two subarcs of $a'$
  bounded by the endpoint of $b$. Then at least one of $b'_+,b'_-$ has length at most
  $L$, say $b'_+$. As a consequence, the arc $a'_+=b\cup b'_+$ has length
  at most $2L$, and thus its homotopy class is contained in $\mathcal{H}$.
\end{proof}
We are now in position to prove Theorem~\ref{thm:fixed_arc}.
\begin{proof}[Proof of Theorem~\ref{thm:fixed_arc}]
  By Lemma~\ref{lem:arcprojection} the assignment of outermost
  surgery pairs satisfies the axioms of a
  dismantling projection.
  Let $\alpha_1$ be an arbitrary arc in $\mathcal{A}^0(X)$, and let
  $\alpha_1,\ldots,\alpha_k$ be the orbit of $\alpha_1$ under $H$.
  Lemma~\ref{lem:arc-length-reduction} shows that
  $\mathcal{H}(\alpha_1,\ldots,\alpha_k)$
  is $\Pi_{\alpha_i}$--convex for each $i=1,\ldots,k$. Finally,
  by Lemma~\ref{lem:arcsurgery-equivariance} the family $\Pi_{\alpha_i}$ is $H$--equivariant. Thus $H$ fixes a clique by
  Proposition~\ref{prop:fixedpoint}(ii).
\end{proof}
\begin{rem}
  We could derive Theorem~\ref{thm:fixed_arc} also directly from
  Proposition~\ref{prop:fixedpoint}(i) by considering a
  different finite invariant set.

  Namely, define a \emph{one-corner broken arc obtained from geodesic arcs
  $a_1,\ldots,a_k$} to be an embedded arc $a$ on $\check{X}$
  of the following form: there are subarcs $a^+,a^-$ of $a_i,
  a_j$ for some $i,j$ such that $a = a^+ \cup a^-$.

  We then define $\widetilde{\mathcal{H}}$ to be the finite set of all
  one-corner broken
  arcs obtained from $a_1,\ldots,a_k$. If the set of homotopy classes of the
  arcs $a_1,\ldots,a_k$ is $H$--invariant, then the set $\widetilde{\mathcal{H}}$ is $H$--invariant as well.
  Furthermore, one can show that $\widetilde{\mathcal{H}}$ is
  $\Pi_{\alpha_i}$--convex for the homotopy classes $\alpha_i$ of $a_i$.
  However, the proof is a bit
  more involved than the simple length argument used for
  Lemma~\ref{lem:arc-length-reduction}, and we decided to omit it.
\end{rem}

\section{Nielsen realisation for surfaces}
\label{sec:Nielsen}

In this section we prove Metatheorem~\ref{met:filling}(2) for mapping class groups (Theorem~\ref{thm:fixed_arc_system}) and Theorem~\ref{thm:Nielsen} under any of the hypotheses (A) or (B).

\subsection{Fixed filling arc set}
A set $A$ of disjoint (simple essential) arcs on a surface $X$ with marked points $\mu$ is
\emph{filling} if the components of $X-(\mu\cup A)$ are discs.
A clique $\Delta\subset \mathcal{A}^0(X)$ is \emph{filling} if some (hence any)
set of disjoint representatives of elements in $\Delta$ is filling.

\begin{thm}
\label{thm:fixed_arc_system}
Let $X$ be a closed connected oriented surface with nonempty
set of marked points and negative Euler characteristic.
Let $H$ be a finite subgroup of $\mathrm{Map}(X)$. Then $H$ fixes a filling clique in $\mathcal{A}^0(X)$.
\end{thm}

\begin{proof}
By Theorem~\ref{thm:fixed_arc}, there is a clique $\Delta\subset \mathcal{A}^0(X)$ fixed by $H$.
Assume that $\Delta$ is a maximal clique fixed by $H$. We will prove that $\Delta$ is filling. We can realise $\Delta$ as a subset $A\subset S$ of disjoint arcs. Equip $X$ with a path-metric in which each arc of $A$ has length $1$. By \cite[Thm~A.5]{Bus} we can also choose representatives of the elements of $H$ fixing $A$ and restricting to a genuine isometric action on $A$.
Note that these representatives satisfy the composition rule up to isotopy fixing $A$.
We will prove that $A$ is filling.

Otherwise, let $X_0$ be the nonempty union of open components of $X-A$ which are not discs, i.e.\ have nonpositive Euler characteristic. Let $\overline{X}_0$ be the path-metric completion of $X_0$. Let $X'$ be the surface obtained from $\overline{X}_0$ by collapsing each component of $\partial \overline{X}_0$ to a marked point. The group $H$ maps into $\textrm{Map}(X')$. Every component of $X'$ with zero Euler characteristic is a sphere with two marked points and contains a unique (homotopy class of) an arc. In the case where $X'$ has such components, we consider the $H$--invariant clique $\Delta'\subset \mathcal{A}^0(X')$ of unique arcs in all of these components. Otherwise an $H$--invariant clique $\Delta'\subset \mathcal{A}^0(X')$ exists by Theorem~\ref{thm:fixed_arc}. Denote $\Delta'=\{\alpha'_1,\ldots,\alpha_k'\}$.

Let $\mathcal{A}^*(X,A)$ denote the set of homotopy classes of (possibly non-simple) essential arcs on $\overline{X}_0$ with endpoints in $\partial \overline{X}_0$. We require homotopies to fix the endpoints. Let $\mathcal{P}\subset \partial \overline{X}_0$ be a family of points, one in each of the components of $\partial\overline{X}_0$.
For $i=1,\ldots,k$ choose arcs $\alpha^*_i(0)\in\mathcal{A}^*(X,A)$ with endpoints $p_i,q_i\in \mathcal{P}$ that project to $\alpha'_i$ on $X'$ and have representatives that are simple and disjoint outside $\mathcal{P}$. For a point $p\in \partial \overline{X}_0,\ t\in \R$ let $\gamma^*(p,t)$ be the homotopy class of the (non-essential) arc starting at $p$, moving clockwise along the boundary component $c\subset \partial \overline{X}_0$ containing $p$ and terminating after distance $t|c|$, where $|c|$ denotes the length of $c$. In particular for $t\in \Z$ the endpoint of $\gamma^*(p,t)$ is $p$. To each point $z=(x_1,y_1,\ldots,x_k, y_k) \in\R^{2k}$ we assign the following sequence $\Delta^*(z)=(\alpha^*_1,\ldots,\alpha^*_k)$ of arcs in $\mathcal{A}^*(X,A)$. Each $\alpha^*_i$ is the homotopy class of the concatenation $\gamma^*(p_i,x_i)^{-1}\alpha^*_i(0)\gamma^*(q_i,y_i)$.

Consider first the case where all of the components of $X'$ have negative Euler characteristic. In that case $\Delta^*$ is injective, since Dehn twists along a pair of components of $\partial \overline{X}_0$ are independent. Note that the set $Z\subset \R^{2k}$ of points $z$ for which $\Delta^*(z)$ can be represented as a collection of disjoint simple arcs is nonempty, since it contains $0$. We claim that $Z$ is convex.

The arc $\alpha^*_i$ might self-intersect only when its endpoints lie on the same component of $\partial \overline{X}_0$, hence only when $p_i=q_i$. Self-intersection of $\alpha^*_i$ arises when starting from $\alpha^*_i(0)$ the two endpoints pass through one another. This is governed, up to interchanging $x_i$ with $y_i$, by the equation
$x_i\leq y_i\leq x_i+1$, which is a linear constraint. The description is the same for intersections of distinct $\alpha^*_i, \alpha^*_j$, where we will get a constraint for each pair of coinciding endpoints of $\alpha^*_i(0), \alpha^*_j(0)$. This justifies the claim.

Under identification of $Z$ with $\Delta^*(Z)$ the finite group $H$ acts on $Z$ by isometries (each of which is a composition of a translation and an interchange of the coordinate axes). Hence $H$ fixes a point $z\in Z$ (for example the centre of mass of some orbit). Let $\varepsilon\in \R^{2k}$ be the vector with smallest nonnegative entries such that $\Delta^*(z+\varepsilon)$ has endpoints in marked points of $X$ under the map $\overline{X}_0\rightarrow X$. Note that $z+\varepsilon$ is also $H$--invariant. The arc set $\Delta^*(z+\varepsilon)$ can be interpreted as a subset of $\mathcal{A}^0(X)$. Then $\Delta\cup \Delta^*(z+\varepsilon)$ forms an $H$--invariant clique larger than $\Delta$ and we have reached a contradiction.

In the case where $\alpha'_i$ lie in the components of $X'$ of zero Euler characteristic we have $\Delta^*(z+e_{2i-1}-e_{2i})=\Delta^*(z)$, where $e_i$ denotes the $i$--th basis vector of $\R^{2k}$. Arcs $\alpha^*_i$ for each of the lines $z+te_{2i-1}-te_{2i}$ where $t\in \R$ form families of parallel arcs. Each such family can be thought of as having the same ``slope'' in the annulus component of $\overline{X}_0$. After quotienting $Z\subset\R^{2k}$ by all $k$ such $\R$--actions we obtain an action of $H$ on a convex subset of $\R^k=\R^{2k}\slash \R^k$ of ``slopes''. In the conclusion of the fixed-point theorem we obtain $H$--invariant ``slope sets'': families of parallel arcs in the components of zero Euler characteristic. We can improve them to families of parallel arcs with endpoints in marked points as before.
\end{proof}

\subsection{Fixed hyperbolic metric}
\begin{proof}[Proof of Theorem~\ref{thm:Nielsen} under hypothesis (A)]
By Theorem~\ref{thm:fixed_arc_system}, the group $H$ fixes a filling clique
$\Delta\subset \mathcal{A}^0(X)$. We can realise $\Delta$ as a set $A$ of disjoint arcs on $X$. For each component $Y$ of $X-A$ with $k$ arcs in its boundary we consider a regular ideal hyperbolic $k$--gon $\widehat{Y}$. For any arc $a\in A$ adjacent to components $Y_1,Y_2$ of $X-A$ we glue

$\widehat{Y}_1$ with $\widehat{Y}_2$ along the sides $\widehat{A}_1\subset \widehat{Y}_1,\widehat{A}_2\subset \widehat{Y}_2$ corresponding to $a$ in such a way that the projections of the centres of  $\widehat{Y}_1$ and $\widehat{Y}_2$ onto $\widehat{A}_1$ and $\widehat{A}_2$ coincide under that gluing. The hyperbolic metric on $X-\mu$ that we obtain as a result is complete and $H$--invariant.
\end{proof}

\begin{proof}[Proof of Theorem~\ref{thm:Nielsen} under hypothesis (B)]
We can assume that the set of marked points on $X$ is empty. Realise the curves of the fixed set as a set $A$ of disjoint simple curves on $X$.
We can choose representatives of elements of $H$ fixing $A$ (see \cite[Thm~A.3]{Bus}).
Let $X'$ be obtained from $\overline{X-A}$ by collapsing its boundary components to marked points. Denote by $\mu'$ the set of marked points on $X'$. By Theorem~\ref{thm:Nielsen} under hypothesis (A) there is a complete hyperbolic metric on $X'-\mu'$ invariant under the image of $H$ in $\mathrm{Map}(X')$. Let $l$ be small enough so that the horocycles of length $l$ around the punctures embed on $X'$. We remove the puncture neighbourhoods bounded by these horocycles. This results in an $H$--invariant complete hyperbolic metric $d_\mathbf{H}$ on $\overline{X-A}$ whose boundary components are horocycles of length $l$. If there are $k$ curves in $A$, the ways of identifying the boundary components of $(\overline{X-A},d_\mathbf{H})$ in order to obtain a (marked) conformal structure on $X$ are parametrised by points in an affine space $\R^k$: changing the $k$--th coordinate corresponds to a twist along the $k$--th boundary component. The group $H$ admits an action on this parameter space $\R^k$ by isometries (which are again compositions of translations and axes interchanges). The fixed point of this action gives an $H$--invariant conformal structure on $X$. By uniformisation this gives an $H$--invariant complete hyperbolic metric on $X$.
\end{proof}

\section{Fixed clique in the disc graph}
\label{sec:fixdisc}
In this section we prove Metatheorem A(1) in the case of the disc graph.
We denote by $U_g$ a handlebody of genus $g$.
A \emph{(disconnected)
  handlebody} $U$ is a disjoint union of various $U_g$.
The boundary $\partial U$ is a closed surface which may be disconnected.
An \emph{essential disc} $D$ in $U$ is a disc properly embedded in
$U$, such that $\partial D$ is an essential simple closed curve on
$\partial U$. We say that an essential simple closed curve $d$ on
$\partial U$ is \emph{discbounding} if there is an essential disc $D$
with $\partial D = d$. Two essential discs $D, D'$ are properly
homotopic if and only if their boundary curves $\partial D$ and
$\partial D'$ are homotopic.

By $\mathcal{D}(U)$ we denote the \emph{disc graph} of
$U$. The vertex set $\mathcal{D}^0(U)$ is the set of homotopy
classes of discbounding curves on $\partial U$. Two such vertices are connected
by an edge if the corresponding curves are disjoint (up to homotopy).

The \emph{handlebody group} $\mathrm{Map}(U)$ is the mapping class
group of $U$, i.e. the group of orientation preserving homeomorphisms
of $U$ up to isotopy. It is well-known that $\mathrm{Map}(U)$ can be
identified with a subgroup of $\mathrm{Map}(\partial U)$ by
restricting homeomorphisms of $U$ to the boundary.

The main theorem of this section is the following.
\begin{thm}
\label{thm:fixed_disc}
Let $U$ be a (possibly disconnected) handlebody
each component of which has genus $\geq 2$.
Let $H$ be a finite subgroup of $\mathrm{Map}(U)$.
Then $H$ fixes a clique in the disc graph $\mathcal{D}(U)$.
\end{thm}
As an immediate consequence we obtain the following.
\begin{cor}
\label{cor:reducib}
  Every finite order element of the handlebody group is reducible.
\end{cor}
Corollary~\ref{cor:reducib} also follows from the Lefschetz fixed point
theorem, since the disc complex is contractible \cite[Thm~5.3]{McC}.

\subsection{Disc surgery}
In this section we describe the surgery procedure which will be used to
define dismantling projections on the disc graph. This surgery procedure was used in \cite{McC} to prove contractibility of the disc complex.

Let $D$ and $D'$ be two transverse essential discs in $U$. We say that $D$ and
$D'$ are in \emph{minimal position}, if the boundary curves $d = \partial D$ and $d' = \partial D'$ are in minimal position, and
$D$ and $D'$ intersect only along arcs.

For any pair of discbounding curves $d$ and $d'$ in minimal position, it is always possible to find discs $D$ and $D'$ in minimal
position with $\partial D = d$, $\partial D' = d'$.
However, at this point we want to warn the reader that the minimal
position of discs is not unique: in particular, the homotopy
classes of $d$ and $d'$ do not determine which
intersection points of $d$ and $d'$ are connected by
an intersection arc of $D\cap D'$.

\medskip
Let now $D$ and $D'$ be two essential discs in minimal position, and
let $d$ and $d'$ denote their boundary curves. Assume that $d$ and
$d'$ intersect. We then say that a subarc $a$ of $d$ is
\emph{outermost in $D$ for $D'$}, if $a$ bounds a disc in
$D$ together with a component
of $D \cap D'$, and the interior of $a$ is disjoint from $D'$.
The endpoints of $a$ decompose $d'$ into two subarcs $a'_+$ and
$a'_-$. The unions $d'_+=a \cup a'_+$ and $d'_-=a \cup a'_-$ are simple closed
curves on $\partial U$, both of which are discbounding by
construction. See e.g. \cite[Sec~5]{HH11} or \cite{Ma86} for a detailed description of
this surgery procedure. Since $d$ and $d'$ are in minimal position, both $d'_+$ and $d'_-$ are essential. We say that the
discbounding curves $d'_+$ and $d'_-$ are obtained by
\emph{outermost surgery of $d'$ in direction of $d$ determined by $a$.}

Note that the homotopy classes $\delta'_+$ of $d'_+$ and $\delta'_-$ of $d'_-$ depend only on the homotopy classes $\delta$ of $d$ and $\delta'$ of $d'$, and the choice of an outermost subarc.
We call the pair $\{\delta'_+, \delta'_-\}$ an \emph{outermost surgery pair
of $\delta'$ in direction of $\delta$}.  Note that there are only finitely many
outermost surgery pairs for any choice of discbounding homotopy
classes $\delta$ and $\delta'$.

If $\delta\neq \delta'$ have disjoint representatives, then we say that the only \emph{outermost surgery pair
of $\delta'$ in direction of $\delta$} is
$\{\delta,\delta\}$, interpreted as a pair
both of whose elements are equal to $\delta$.
For homotopy classes $\delta,\delta'\in\mathcal{D}^0(U)$ of
discbounding curves we define $\Pi_\delta(\delta')$ to be the set of
all outermost surgery pairs of $\delta'$ in direction of $\delta$.

\medskip
The next lemma shows that $\Pi_\sigma$ satisfies the axioms of a
$\sigma$--projection for each $\sigma\in\mathcal{D}^0(U)$.
This lemma essentially follows from the proof of \cite[Thm~5.3]{McC}.
\begin{lem}\label{lem:discsurgery}
  Let $\sigma\in\mathcal{D}^0(U)$ be a discbounding curve on $\partial U$.
  \begin{enumerate}[(i)]
  \item Let $R \subset \mathcal{D}^0(U)$ be a finite set of
    discbounding curves with $R \setminus\{\sigma\}\neq\emptyset$. Then
    there is a discbounding curve $\rho \in R\setminus\{\sigma\}$ with the
    following property: there is an outermost surgery pair
    $\{\rho_+,\rho_-\}$ of $\rho$ in direction of $\sigma$ such that
    each discbounding curve $\rho'\in R$ which is disjoint from $\rho$
    is also disjoint from $\rho_+$ and $\rho_-$.
  \item Every sequence $(\rho_i)$ of discbounding curves in
    $\mathcal{D}^0(U)$ such that $\rho_{i+1}$ is contained in an
    outermost surgery pair of $\rho_i$ in direction of $\sigma$
    terminates after finitely steps with the discbounding curve $\sigma$.
  \end{enumerate}
\end{lem}
Again, we will use hyperbolic geometry to choose preferred
representatives in the homotopy class of discbounding curves. To this
end we fix a hyperbolic metric on $\partial U$.
\begin{proof}
  \begin{enumerate}[(i)]
  \item Note that if each element of $R$ can be realised
    disjointly from $\sigma$ then the assertion follows trivially.

    Otherwise, let $s$ be the geodesic representative of $\sigma$ and let
    $r^1$ be the geodesic representative of an arbitrary
    element in $R$ intersecting $s$.
    Choose an
    embedded disc $D^1$ in $U$ bounded by
    $r^1$ and an embedded disc $D$ bounded by $s$ which are in minimal
    position.
    Let $a\subset \partial D$ be an outermost subarc in $D$ for
    $D^1$. Let $N$ be the set of geodesic representatives of all the elements
    in $R$ which are disjoint from $r^1$.
    For each $r' \in N$ we may choose a disc $D(r')$ bounded by $r'$
    which is disjoint
    from $D^1$ and in minimal position with respect to $D$.
    Note that this means that outermost subarcs in $D$ for
    the discs $D(r')$ are either disjoint from $a$, contain $a$, or are
    contained in $a$. We distinguish two cases.

    If all outermost subarcs in $D$ for the discs $D(r'), r' \in N$ are
    disjoint from $a$ or contain
    $a$, then both discbounding curves obtained by the outermost surgery pair determined by $a$ are
    disjoint from the curves in $N$, and therefore the homotopy class
    $\rho^1$ of $r^1$ satisfies the condition in assertion (i).

    On the other hand, assume that there is a disc $D^2=D(r^2)$ bounded by $r^2 \in
    N$ which has an outermost subarc that is contained in $a$. In this
    case, we replace $r^1$ by $r^2$, and inductively
    apply the same argument again.
    Since the geodesic representatives of all elements in $R$
    intersect in finitely many points, this
    process stops with the first case after finitely many steps with a $\rho^k$
    satisfying the condition in assertion (i).
  \item The proof of this is identical to the proof of
    Lemma~\ref{lem:arcprojection}(ii). Namely, each outermost
    surgery strictly decreases the geometric intersection number with
    $\sigma$.
  \end{enumerate}
\end{proof}
The fact that homeomorphisms of $U$ map outermost subarcs to outermost
subarcs immediately implies the following.
\begin{lem}\label{lem:discsurgery-equivariance}
  For every $h \in \mathrm{Map}(U)$ we have
  $$h\Pi_\sigma(\rho) = \Pi_{h\sigma}\left(h\rho\right)$$
  for all discbounding curves $\sigma\neq\rho \in \mathcal{D}^0(U)$.
\end{lem}

\subsection{Finite hull for discs}

Let $\delta_1,\ldots,\delta_k\in \mathcal{D}^0(U)$ be a finite set of discbounding
curves. For a homotopy class $\alpha$ of an essential simple closed curve on $\partial U$, we denote by $l_{\partial U}(\alpha)$
the length of the geodesic representative of $\alpha$ on $\partial U$. Put
$$L = \max_i l_{\partial U}(\delta_i).$$
Let $\mathcal{H}=\mathcal{H}(\delta_1,\ldots,\delta_k)$ be the set of discbounding curves
$\delta$ which satisfy $l_{\partial U}(\delta) \leq 2L$.
Note that $\mathcal{H}(\delta_1,\ldots,\delta_k)$ is a finite set of curves
which contains $\delta_1,\ldots,\delta_k$.

The following lemma shows that $\mathcal{H}$ is $\Pi_\sigma$--convex
for each $\sigma=\delta_i, i=1,\ldots,k$.
\begin{lemma}\label{lem:disc-hulls}
  Let $\delta_1,\ldots,\delta_k\in\mathcal{D}^0(U)$ be a finite set of
  discbounding
  curves. Put $\delta = \delta_i$ for some $i=1,\ldots, k$, and let $\delta'
  \in \mathcal{H}=\mathcal{H}(\delta_1,\ldots,\delta_k)$ be arbitrary. Then
  each outermost surgery pair of $\delta'$ in direction of $\delta$
  contains at least one element of $\mathcal{H}$.

\end{lemma}
\begin{proof}
  Denote by $d$ the geodesic representative of $\delta$, and by $d'$
  the geodesic representative of $\delta'$. We can assume that $d$ and $d'$ intersect.
  Let $a\subset d$ be an outermost subarc.
  By definition of $L$, the subarc $a$ has length at most $L$.

  The endpoints of $a$ decompose $d'$ into two subarcs $a'_+$ and
  $a'_-$. Since $d'$ has length at most $2L$, at least one of these
  subarcs has length less or equal to $L$, say $a'_+$. Then
  $d'_+=a \cup a'_+$ has length at most $2L$. In particular, the homotopy
  class of
  $d'_+$ is contained in $\mathcal{H}$ as required.
\end{proof}

As a consequence of Lemmas~\ref{lem:discsurgery}, \ref{lem:disc-hulls}
and \ref{lem:discsurgery-equivariance} we obtain Theorem~\ref{thm:fixed_disc}
from Proposition~\ref{prop:fixedpoint}(ii) in the same
way that we obtained Theorem~\ref{thm:fixed_arc}.

\section{Nielsen realisation for handlebodies}
\label{sec:Nielsen_for_handle}

\subsection{Fixed simple disc system}
Let $\{d_1,\ldots,d_k\}$ be a finite set of discbounding curves
which are pairwise disjoint. We say that $\{d_1,\ldots,d_k\}$ is
a \emph{simple disc system}, if each complementary component of
$d_1 \cup \dots \cup d_k$ on $\partial U$ is a \emph{bordered}
$2$--sphere, i.e.\ a $2$--sphere minus open discs. Equivalently,
the complementary components of disjoint discs bounded by
$d_1,\ldots,d_k$ in $U$ are simply connected.
A clique $\Delta \subset \mathcal{D}^0(U)$ is called \emph{simple},
if some (hence any) set of disjoint representatives of elements in $\Delta$ is
a simple disc system.

In this section we use the results obtained in
Section~\ref{sec:fixdisc} to show the following.
\begin{thm}\label{thm:simple-fixed}
\label{thm:fixed_disc_system}
Let $U$ be a connected handlebody of genus $g \geq 2$. Let $H$ be a finite subgroup of $\mathrm{Map}(U)$. Then $H$ fixes a
   simple clique in the disc graph $\mathcal{D}^0(U)$.
\end{thm}
Before we give a proof, we note the following corollary.
\begin{cor}
\label{cor:char_conj_into_handlebody}
  Let $U$ be a connected handlebody of genus $g \geq 2$. A finite order
  mapping class $h\in \mathrm{Map}(\partial U)$ is conjugate into the handlebody
  group $\mathrm{Map}(U)$ if and only if $h$ fixes a set of homotopy classes of disjoint essential simple closed curves on $\partial U$ all of whose complementary components
  are bordered $2$--spheres.
\end{cor}
\begin{proof}
  If $h$ is contained in a conjugate of the handlebody group, then
  Theorem~\ref{thm:simple-fixed} yields the desired set of curves. Let now
  $h$ be a mapping class fixing a set of curves
  $\{\alpha_1,\ldots,\alpha_k\}$ as required. We may then choose a
  mapping class $g\in \mathrm{Map}(\partial U)$ such that $g(\alpha_i)$ are discbounding in
  $U$. This element $g$ then conjugates $h$ into the handlebody group
  of $U$.
\end{proof}

\begin{proof}[Proof of Theorem~\ref{thm:simple-fixed}]
  By Theorem~\ref{thm:fixed_disc} there is a clique
  $\Delta=\{\delta_1,\ldots,\delta_n\}$ in the disc graph
  $\mathcal{D}(U)$ invariant under $H$. We will argue that if $\Delta$ is
  a maximal $H$--invariant clique, then it is simple.

  As a first step, using Theorem~\ref{thm:Nielsen} with hypothesis
  (B), we find that there is a hyperbolic metric on $X=\partial U$ such
  that $H$ acts as a group of isometries on $X$.
  Denote by $d_i$ the geodesic representative of $\delta_i$. Then $H$
  acts on $X$ preserving the set $d_1\cup\dots\cup d_n$.
  Let $X_0$ be the union of all the components of $X - (d_1 \cup
  \dots \cup d_n)$ which are not bordered $2$--spheres.
  Denote by $\overline{X}_0$ the completion (with respect to the path metric) of
  $X_0$. We will prove that $\overline{X}_0$ is empty, which implies that $\Delta$ is simple.

  Suppose on the contrary that $\overline{X}_0$ is nonempty. In this case, let
  $X'$ be the surface obtained from $\overline{X}_0$ by gluing a disc
  to each boundary component. Since $\overline{X}_0$ embeds naturally into
  $X'$, we can identify $\overline{X}_0$ with a subsurface of
  $X'$.
  Since $H$ acts as a group of isometries
  on $\overline{X}_0$, it maps into the mapping class
  group of $X'$. Furthermore, $X'$ is in a natural way a (possibly
  disconnected) handlebody $U'$ (cut $U$ along discs
  bounded by $d_i$) and so $H$ maps into the handlebody group of $U'$.

  If $U'$ has genus $1$ components, then we consider the $H$--invariant clique $\Delta'\subset \mathcal{D}^0(U')$ of unique discbounding curves in all of these components. Otherwise the group $H$ fixes a clique
  $\Delta'\subset \mathcal{D}^0(U')$ by Theorem~\ref{thm:fixed_disc}.
Denote $\Delta'=\{\delta'_1, \ldots, \delta'_k\}$.

  For a homotopy class of an essential simple closed curve $\alpha'$ on $X'$
  let $\mathcal{F}(\alpha')$ be the set of all homotopy classes of
  simple closed curves $\alpha$ on $\overline{X}_0$ which are homotopic to
  $\alpha'$ as a curve on $X'$. Note that each such
  $\alpha$ is essential. Furthermore, if $\alpha'$ is discbounding in
  $U'$, then each element of $\mathcal{F}(\alpha')$ is discbounding in
  $U$. We put $\mathcal{F}(\Delta') = \mathcal{F}(\delta_1') \cup \dots \cup
  \mathcal{F}(\delta_k')$.

  Set
  $$L = \min\{l_{\overline{X}_0}(\alpha) | \ \alpha \in \mathcal{F}(\Delta') \}$$
  and
  $$\Delta^* = \{ \alpha \in \mathcal{F}(\Delta')  |\
  l_{\overline{X}_0}(\alpha) = L\}.$$
  The group $H$, seen as a subgroup
  of the mapping class group of $\overline{X}_0$, preserves $\Delta^*$.
  We claim that any two elements of $\Delta^*$ correspond to
  disjoint curves.
  This claim implies the theorem, since $\Delta^*$ can be then interpreted as a clique in $\mathcal{D}(U)$. Thus we can extend the initial clique
  $\Delta$ by the clique $\Delta^*$, contradicting maximality.

  To prove the claim, let $\alpha_1,\alpha_2$ be two elements of
  $\Delta^*$, and let $a_1$ and $a_2$ be their geodesic representatives
  on $\overline{X}_0$. Suppose that $a_1$ and $a_2$ intersect. Since $a_1$ and
  $a_2$ are homotopic to disjoint curves after gluing discs to the
  boundary components of $\overline{X}_0$, there are subarcs $c_i \subset a_i$
  such that $c_1 \cup c_2$ bounds a bordered $2$--sphere together with
  boundary components of $\overline{X}_0$ (see Figure~\ref{fig:discreduction}).
  \begin{figure}[h!]
    \centering
    \includegraphics[width=0.35\textwidth]{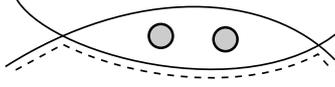}
    \caption{Intersecting preimages are not shortest}
    \label{fig:discreduction}
  \end{figure}
  Suppose that the length of the geodesic arc $c_1$ is smaller (or
  equal to) the length of $c_2$.
  In this case, the curve $\hat{a}$ obtained from $a_2$ by replacing
  $c_2$ with $c_1$ is a broken geodesic which is homotopic to $a_2$ on
  $X'$. Hence its homotopy class lies in $\mathcal{F}(\Delta')$. However,
  the length of the broken geodesic $\hat{a}$ is by construction at most
  $L$, and hence its geodesic representative has length strictly less
  than $L$. This contradicts the choice of $L$, and therefore
  proves the claim.
\end{proof}

\subsection{Fixed Schottky group}
In this section we obtain a Nielsen realisation theorem for the
handlebody group.

\begin{proof}[Proof of Theorem~\ref{thm:schottky-nielsen}]
  Let $X = \partial U$. By Theorem~\ref{thm:simple-fixed}
  there is a simple clique $\Delta=\{\delta_1,\ldots,\delta_n\}$ which
  is fixed by the group $H$. Let $\{d_1,\ldots,d_n\}$ be a set of
  disjoint representatives of the $\delta_i$ and denote by $X_1,\ldots,X_k$ the
  complementary components of the $d_i$ on $X$. For each such
  bordered $2$--sphere $X_i$, denote by $X'_i$ the punctured
  sphere obtained by replacing each boundary component of $X_i$ by a
  puncture. Let $X'$ be the disjoint union of these
  punctured spheres. Note that $H$ acts on $X'$ as a group
  of mapping classes. By Theorem~\ref{thm:Nielsen} with hypothesis (A) there is a
  complete hyperbolic metric on $X'$ such that $H$ acts as a group
  of isometries.

  The hyperbolic metric on $X'_i$ can be interpreted as a conformal structure
  on a punctured sphere. By the uniformisation theorem, it is biholomorphic to the complement of a finite
  number of points $p_i^1, \ldots, p_i^{k_i}$ on a copy
  $\hat{\C}_i$ of the Riemann sphere, with $k_i\geq 3$.
  The group $H$ acts on the union $X'$ of $X_i'\subset\hat{\C}_i$ by M\"obius
  transformations, since every conformal automorphism of $\hat{\C}$ is a M\"obius transformation.

  Consider now a round disc $B_i^j$ centred at one of the points
  $p_i^j$, which is disjoint from all other such points. Suppose that $h
  \in H$ fixes the point $p_i^j$. Since $h$ is a M\"obius
  transformation of finite order which fixes $p_i^j$ and
  permutes $k_i-1\geq 2$ other points $p^{j'}_i$, it is in fact a
  rotation about $p_i^j$ and thus preserves $B_i^j$.
 Therefore we may choose a collection of disjoint discs $B_i^j$, such that
  $B_i^j$ is centred at $p_i^j$, which is invariant under the action
  of $H$.

  We interpret each $\hat{\C}_i$ as the boundary of a copy
  $\H^3_i$ of the hyperbolic $3$--space.
  Let $\mathcal{H}(B_i^j)$ be the halfspace in $\H^3_i$ which
  is the convex hull of $B_i^j$, and let $M_i$ be the complement of the interiors of all
  $\mathcal{H}(B_i^j)$ in $\H^3_i$.
Note that the boundary components of the manifolds $M_i$ correspond to punctures of $X'$. The punctures of $X'$ have
  a natural pairing, since each marked point of $X'$
  corresponds to a side of a disc in the simple disc system $\{d_1,\ldots, d_n\}$.
  We now glue the manifolds $M_i$ along their boundaries according to
  this pairing, such that the closest projections of the points
  $p^j_i$ to the boundary planes $\partial\mathcal{H}(B_i^j)$ agree.
  By arguing as in the proof of
  Theorem~\ref{thm:Nielsen} under hypothesis (B) we can choose the
  twist parameters consistently such that $H$ acts on the resulting hyperbolic manifold $M$ as a
  group of isometries.
\end{proof}

\section{Fixed clique in the sphere graph}
\label{sec:spheres}
Let $U$ be a (possibly disconnected) handlebody and let
$W=W(U)$ be the $3$--manifold obtained by
doubling $U$ along its boundary.
The manifold $W$ is homeomorphic to the
disjoint union of iterated connected sums of $S^1 \times S^2$.
A connected component of $W$ has \emph{rank $n$} if its fundamental group has rank $n$.
By $\mathrm{Map}(W)$ we denote the
\emph{mapping class group of $W$}, i.e.\ the group of orientation preserving
homeomorphisms of $W$ up to isotopy.

A $2$--sphere embedded in $W$ is called \emph{essential}, if it does
not bound a ball.
Unless stated otherwise, we assume that spheres are embedded and
essential $2$--spheres.

We define the \emph{sphere graph} $\mathcal{S}(W)$ of $W$. Its vertex set $\mathcal{S}^0(W)$ is the set of isotopy classes
of spheres in $W$. Two vertices of $\mathcal{S}(W)$ are connected by an edge, if the
corresponding spheres can be realised disjointly.

The fundamental group of a doubled connected handlebody $W_n = W(U_n)$
of rank $n$ is the free group on $n$
generators. Hence, the mapping class group $\mathrm{Map}(W_n)$ of $W_n$
admits a homomorphism to $\mathrm{Out}(F_n)$. By a theorem of
Laudenbach \cite[{Thm~4.3, Rem~1}]{L74}, this map
is surjective and has finite kernel generated by Dehn twists along
spheres. Therefore the elements of the kernel act trivially on isotopy classes
of spheres. As a consequence, $\mathrm{Out}(F_n)$ acts as a group
of automorphisms on $\mathcal{S}(W_n)$ (see \cite{Ha} for a thorough
treatment of this graph).

We can now state the main theorem of this section.
\begin{thm}\label{thm:fixedsphere}
  Let $W$ be a (possibly disconnected) doubled handlebody, each component of which has rank $\geq 2$.
  Let $H$ be a finite subgroup of $\mathrm{Map}(W)$. Then $H$ fixes a
  clique in the sphere graph $\mathcal{S}(W)$.
\end{thm}

By the discussion in the previous paragraph
Theorem~\ref{thm:fixedsphere} immediately implies
that if $H$ is a finite subgroup of $\mathrm{Out}(F_n)$, then it fixes a
clique in $\mathcal{S}(W_n)$.

The strategy to prove Theorem~\ref{thm:fixedsphere} will be analogous to the strategy
employed in Sections~\ref{sec:fixarc} and~\ref{sec:fixdisc}. In the cases of the arc and disc graph we used hyperbolic metrics
on surfaces to select preferred representatives of arcs and
discbounding curves. In the manifold $W$ we do not have such geometric tools.
In the following section we recall and extend the notion of \emph{normal
  position} from \cite{Ha}, which is a topological tool allowing to
put pairs of spheres in a preferred position with good properties.

\subsection{Normal position}
\label{sec:normal}

A \emph{sphere system} $S_0$ in $W$ is a collection of disjoint
spheres in $W$, no two of which are isotopic.
Note that two sphere systems are isotopic if and only if they are
homotopic by a theorem of Laudenbach (see \cite{L74}).
A sphere system is
\emph{maximal} if it is not properly
contained in another sphere system.

For a sphere system $S_0$ in $W$ we call the
(path-metric) closures of the connected components of $W-S_0$ the
\emph{complementary components of $S_0$} (in $W$).
A sphere system is maximal if and
only if each of its complementary components in $W$ is homeomorphic
to the $3$--sphere minus three open balls.
Associated to a sphere system $S_0$ is a \emph{dual graph} $\Gamma$ in
the following way. The vertex set of $\Gamma$ is the set of
complementary components of $S_0$. Each sphere $s_0\in S_0$ defines an
edge in $\Gamma$, connecting the vertices corresponding to the two
(possibly coinciding) complementary components of $S_0$ adjacent to $s_0$. Note that the graph $\Gamma$ might not be simple. In particular, it might happen that a complementary component of $S_0$ is adjacent along two distinct boundary components to the same sphere in $S_0$, and then $\Gamma$ contains a loop.

Let $s$ be a sphere which is transverse to $S_0$.
The intersection of $s$ with a complementary component of $S_0$ is a
disjoint union of bordered $2$--spheres. We call these components
\emph{sphere pieces of $s$ with respect to $S_0$}.

A sphere $s$ is in \emph{normal
  position} with respect to a maximal sphere system $S_0$ (see \cite{Ha}) if $s$ is transverse to $S_0$ and each sphere piece $P$ of $s$ satisfies the following two conditions:
\begin{enumerate}[(1)]
\item The piece $P$ meets each boundary component of a complementary
  component of $S_0$ in at most one circle.
\item The piece $P$ is not a disc which is homotopic into
  $S_0$ relative to its boundary.
\end{enumerate}
We say that a sphere piece $P$ is \emph{problematic of type (1) or
  (2)}, if it violates condition (1) or
(2) above. In that case we say that it is \emph{problematic with
  respect to $s_0\in S_0$} if the boundary component of the
complementary component on which $P$ shows the
excluded behaviour corresponds to $s_0$.
We say that a sphere system $S$ is in \emph{normal position} with
respect to $S_0$ if each sphere $s\in S$ is in normal position with
respect to $S_0$.

Normal position of spheres is unique in the following sense:
suppose that $s,s'$ are isotopic spheres which are in
normal position with respect to some maximal sphere system $S_0$.
Then by \cite[Prop~1.2]{Ha} there is a
homotopy between $s$ and $s'$ which restricts to an isotopy on
$S_0$ (Hatcher considers only the case of connected $W$, but the
disconnected case follows immediately by considering the connected components
individually). We say that sphere pieces of $s$ and $s'$ are
\emph{corresponding}, if they are mapped to each other by such a
homotopy. Correspondence identifies the sphere pieces of $s$
bijectively with the sphere pieces of $s'$.

Next, we characterise normal position using lifts to the universal
cover $\widetilde{W}$ of $W$. Since $W$ is allowed to be disconnected,
we define $\widetilde{W}$ to be the disjoint union of the universal
covers of the components of $W$.
Note that since spheres are
simply connected, every sphere in $W$ lifts homeomorphically to
$\widetilde{W}$.
The definitions of sphere systems, complementary components, sphere
pieces and normal position extend verbatim to spheres in $\widetilde{W}$.
\begin{lem}\label{lem:characterizing-normal-position-systems}
  Let $S_0$ be a maximal sphere system in $W$ and let $s$ be a sphere in $W$. Denote
  by $\widetilde{S}_0$ the full preimage of $S_0$ in $\widetilde{W}$.
  Then the following are equivalent:
  \begin{enumerate}[(i)]
  \item The sphere $s$ is in normal position with respect to $S_0$.
  \item One (and hence any) lift $\widetilde{s}$ of $s$ is in normal position with
    respect to $\widetilde{S}_0$.
  \item \begin{enumerate}[(1)]\item One (and hence any) lift $\widetilde{s}$ of $s$ intersects each component of
    $\widetilde{S}_0$ in at most one circle, and
    \item there is no disc in $\widetilde{s}$ homotopic relative to its boundary into
    $\widetilde{S}_0$.
  \end{enumerate}
  \end{enumerate}
\end{lem}
\begin{proof}
  Assertion (i) is equivalent to (ii) since a piece $\widetilde{P}$ of $\widetilde{s}$ that is a lift of piece $P$ of $s$ is problematic if and only if $P$ is problematic.
  It is clear that {(iii)} implies {(ii)}. It remains to prove that assertion {(ii)} implies {(iii)}.

  For part (1) let $\widetilde{s}_0$ be a sphere of
  $\widetilde{S}_0$. Suppose that $\widetilde{s}_0$ and
  $\widetilde{s}$ intersect in at least two circles $\widetilde{c}$
  and $\widetilde{c}'$. Let $a \subset\widetilde{s}$ be a path joining
  $\widetilde{c}$ and $\widetilde{c}'$ and crossing the least number
  of complementary components of $\widetilde{S}_0$. Since the graph
  dual to $\widetilde{S}_0$ in $\widetilde{W}$ is a forest, there is a
  complementary component $\widetilde{K}$ of $\widetilde{S}_0$ such
  that $a$ intersects only one boundary component of
  $\widetilde{K}$. Then a sphere piece of $\widetilde{s}$ in
  $\widetilde{K}$ is problematic of type (1).

  For part (2) suppose there is a disc in $\widetilde{s}$ homotopic relative to its boundary into
  $\widetilde{S}_0$. Then it contains a subdisc which is a piece of $\widetilde{s}$ problematic of type (2).

\end{proof}

Motivated by Lemma~\ref{lem:characterizing-normal-position-systems},
we say that two transverse spheres $\widetilde{s}$ and
$\widetilde{s}'$ in $\widetilde{W}$ are
in \emph{normal position} if
\begin{enumerate}[(1)]
\item  $\widetilde{s}$ and
$\widetilde{s}'$ intersect in at most one circle, and
\item
none of the discs in
$\widetilde{s}-\widetilde{s}'$ is homotopic
into $\widetilde{s}'$ relative to its boundary.
\end{enumerate}
Note that this relation is symmetric with respect to $\widetilde{s}$ and $\widetilde{s}'$.

We say that two spheres $s, s'$ in $W$ are in \emph{normal position}
if all of their lifts to $\widetilde{W}$ are in normal position.
Note that given two spheres $s, s'$ we can change $s$ by a homotopy to
be in normal position with $s'$. Namely, extend $s'$ to a maximal
sphere system $S'$. Then \cite[Prop~1.1]{Ha} implies that we can
homotope $s$ to be in normal position with respect to $S'$. By
Lemma~\ref{lem:characterizing-normal-position-systems} the spheres
$s,s'$ are then in normal position as well. Similarly, given a sphere $s'$ in $W$ we can homotope any sphere system $S$ in $W$ so that $s$ and $s'$ are in normal position for any $s\in S$.

We need the following uniqueness property of normal position for a
pair of spheres.
\begin{lem}\label{lem:uniqueness-normal-position}
  Let $s_0$ be a sphere in $W$, and let $s,s'$ be two isotopic spheres
  in $W$ which are in
  normal position with respect to $s_0$.
  Then there is a homotopy between $s$ and $s'$ which restricts to an
  isotopy on $s_0$.
\end{lem}
\begin{proof}
  Extend $s_0$ to a maximal sphere system $S_0$.  We first claim
  that we can change $s$ and $s'$ to be in normal
  position with respect to $S_0$ by a homotopy which restricts to
  the identity on $s_0$. We give the proof of the claim for $s$, the
  proof for $s'$ is identical.

  Suppose that $s$ is not in
  normal position with respect to $S_0$. Then there is a sphere piece $P$
  of $s$ with respect to $S_0$ which is problematic.
Following the proof of \cite[Prop~1.1]{Ha} we may modify $s$
  by a homotopy to remove this problematic piece.
  This homotopy moves $s$ through $s_0$ if and only if the
  sphere piece $P$ is problematic with respect to $s_0$.
  If $P$ was problematic of type (1), then a lift of $s$ would
  intersect a lift of $s_0$ in two circles. If $P$ was problematic of
  type (2), then a lift of $s$ would contain a disc which is
  homotopic into a lift of $s_0$ relative to its boundary.
  Both of these cases are excluded by normal position of $s$ and $s_0$.
Hence every problematic
  piece $P$ of $s$ with respect to $S_0$ can be removed by a
  homotopy of $s$ which does not move $s$ through $s_0$.
  Furthermore, such homotopies keep $s$ in normal position with
  respect to $s_0$. This shows the claim.

  Now the assertion of the lemma follows from
  \cite[Prop~1.2]{Ha} which states that there is a homotopy between $s$ and $s'$
  which restricts to an isotopy on $S_0$.
\end{proof}

\subsection{Sphere surgery}
\label{sec:sphere-surgery}
In this section we describe a surgery procedure for spheres. This
surgery procedure was defined and used in \cite{Ha} to show
contractibility of the sphere complex.

Suppose that spheres $s$ and $s'$ in $W$ intersect and are in normal position.
We say that a disc $D \subset s$ is
\emph{outermost in $s$ for $s'$} if  $D\cap s'=\partial D$.

Let $D$ be an outermost disc in $s$ for $s'$. Denote by $D'_+$ and $D'_-$ the two components of $s' -
\partial D$. Then we say that the spheres $s'_+ = D'_+ \cup D, s'_- = D'_- \cup
D$ are obtained by \emph{outermost surgery of $s'$ in direction of $s$ determined by $D$}.
Normal position of $s$ and $s'$ implies that both $s'_+$ and $s_-'$
are essential spheres.

By Lemma~\ref{lem:uniqueness-normal-position} the notion of outermost
discs is well-defined for isotopy classes of
spheres. Hence the isotopy classes $\sigma'_+,\sigma'_-$ of $s'_+,s'_-$
depend only on the isotopy classes $\sigma$
of $s$ and $\sigma'$ of $s'$.
We call the pair $\{\sigma'_+,\sigma'_-\}$ an
\emph{outermost surgery pair of $\sigma'$ in direction of $\sigma$}.

If $\sigma'\neq \sigma$ are disjoint, then the only \emph{outermost surgery pair of $\sigma'$ in direction of $\sigma$} is $\{\sigma,\sigma\}$,
interpreted as a pair both
of whose elements are equal to $\sigma$.
Let $\Pi_\sigma(\sigma')$ be the set of
all outermost surgery pairs of $\sigma'$ in direction of $\sigma$.

\medskip
We next show that $\Pi_\sigma$ satisfies the axioms of a
$\sigma$--projection for all $\sigma \in \mathcal{S}^0(W)$.
The main technical result needed for this is given by the following
lemma.
\begin{lemma}\label{lem:no-sphere-cycles}
  Let $S_0$ be a maximal sphere system in $W$. Suppose that $r_1,\ldots,
  r_k$ is a sequence of spheres in normal position with respect to
  $S_0$. Assume that for each $1\leq i\leq k-1$ the sphere
  $r_{i+1}$ is disjoint from $r_i$ and that $r_1$ is isotopic to $r_k$.
Let $c_i \subset r_i \cap S_0$ be a sequence of intersection
  circles bounding discs $D_i \subset S_0$.
  Suppose that $D_{i+1} \subset D_i$ for each $1\leq i\leq k-1$.
  If $c_1$ and $c_k$ are boundaries of corresponding sphere
  pieces of $r_1$ and $r_k$, then all $r_i$ are in fact isotopic.
\end{lemma}
\begin{proof}
  Suppose that $r_i$ are spheres satisfying the hypothesis of the
  lemma.
Denote by $K$ the complementary component of $S_0$ which
  contains the corresponding sphere pieces $P_1$ of $r_1$ and $P_k$ of $r_k$.
  For each $1< i<k$, let $P_i$ be the sphere piece of $r_i$
  contained in $K$ and bounded by $c_i$.
 The complementary component $K$ lifts into $\widetilde{W}$ homeomorphically to a
  complementary component $\widetilde{K}$ of the full preimage $\widetilde{S}_0$ of $S_0$.
  Let $\widetilde{P}_i,\widetilde{D}_i$ be the lifts of $P_i,D_i$ into
  $\widetilde{K}$. Let $\widetilde{r}_i$ be the lifts of the $r_i$ containing
  $\widetilde{P}_i$.

  Each sphere $\widetilde{r}_i$ separates the universal cover
  $\widetilde{W}$ into two connected components. We call the component
  which contains $\widetilde{D}_i$ the \emph{outside} of
  $\widetilde{r}_i$.
  Since $r_i$ and $r_{i+1}$ are disjoint, the lift $\widetilde{r}_{i+1}$ is
  completely contained on the outside of $\widetilde{r}_i$  for each
  $1\leq i\leq k-1$. In particular, $\widetilde{r}_1$ and
  $\widetilde{r}_k$ are disjoint.

  By the definition of correspondence, there is a homotopy of $r_1$ to
  $r_k$ which restricts to an isotopy on $S_0$ and which maps $P_1$ to
  $P_k$. Lifting this homotopy we see that $\widetilde{r}_1$ is
  homotopic to $\widetilde{r}_k$. Therefore, $\widetilde{r}_1$ and
  $\widetilde{r}_k$ bound a product region which contains all the
  $\widetilde{r}_i$. In particular, all $\widetilde{r}_i$ are
  homotopic, as desired.
\end{proof}

\begin{lem}\label{lem:sphere-projection}
  Let $\sigma\in\mathcal{S}^0(W)$ be a sphere in $W$.
  \begin{enumerate}[(i)]
  \item Let $R \subset \mathcal{S}^0(W)$ be a finite set of
    spheres with $R \setminus\{\sigma\}\neq\emptyset$. Then
    there is a sphere $\rho \in R\setminus\{\sigma\}$ with the
    following property: there is an outermost surgery pair
    $\{\rho_+,\rho_-\}$ of $\rho$ in direction of $\sigma$ such that
    each sphere $\rho'\in R$ which is disjoint from
    $\rho$ is also disjoint from $\rho_+$ and $\rho_-$.
  \item Every sequence $(\rho_i)$ of spheres such that $\rho_{i+1}$ is
    obtained from $\rho_i$ by an outermost surgery in direction of
    $\sigma$ terminates after finitely many steps with $\sigma$.
  \end{enumerate}
\end{lem}
\begin{proof}
  \begin{enumerate}[(i)]
  \item We can assume that not all the spheres in $R$ are disjoint from $\sigma$, since otherwise the assertion follows trivially. Let $s$ be a representative of $\sigma$ and let $S$ be an
    extension of $s$ to a maximal sphere system.
    Let $r^1$ be a representative of an arbitrary element $\rho^1\in R$ intersecting $s$ in normal
    position with respect to $S$.
    Let $D\subset s$ be an outermost disc for $r^1$ and let $N$ be the set
    of representatives of the spheres in $R\setminus \{\rho^1\}$ which are disjoint from $r^1$.

    We may assume that each $r' \in N$ is in normal position with respect to $S$.
    Note that this means that outermost discs in $s$ for
    the spheres $r'$ are either disjoint from $D$, contain $D$, or are
    contained in $D$. We distinguish two cases.

    If all outermost discs in $s$ for all $r' \in N$ are
    disjoint from $D$ or contain
    $D$, then both spheres obtained by outermost surgery determined by $D$ are
    disjoint from the spheres in $N$, and therefore $r^1$
    satisfies the condition required in {(i)}.

    On the other hand, assume that there is a sphere $r^2 \in
    N$ for which an outermost disc is contained in $D$. In this
    case, we replace $r^1$ by $r^2$, and inductively apply the same
    argument again.
    By Lemma~\ref{lem:no-sphere-cycles}, this process terminates after
    finitely many steps with the first case.
  \item The proof of this part is similar to the proof of
    Lemma~\ref{lem:arcprojection}(ii).
    Namely, suppose that $s$ and $s'$ are two spheres in normal
    position which intersect in $k$ circles. Then both spheres obtained by an
    outermost surgery of $s'$ in direction of $s$ intersect $s$
    in at most $k-1$ circles. Putting spheres in normal position
    decreases the number of intersection circles (see the proof of
    \cite[Prop~1.1]{Ha}). Thus an induction on this value proves the assertion.
  \end{enumerate}
\end{proof}
The fact that homeomorphisms of $W$ map outermost discs to outermost discs immediately implies the following.
\begin{lem}\label{lem:spheresurgery-equivariance}
  For every $h \in \mathrm{Map}(W)$ we have
  $$h\Pi_\sigma(\rho) = \Pi_{h\sigma}\left(h\rho\right)$$
  for all spheres $\sigma\neq\rho$ in $\mathcal{S}^0(W)$.
\end{lem}

\begin{rem}
  Let $\sigma$ be the isotopy class of a sphere in $W$.
  In \cite{HH11b} it is shown that the stabiliser of $\sigma$ in
  $\mathrm{Map}(W)$ is an undistorted subgroup of
  $\mathrm{Map}(W)$. Using the techniques developed in this section,
  the proof in \cite{HH11b} can be improved to show that the stabiliser of $\sigma$ in
  $\mathrm{Map}(W)$ is in fact a
  coarse Lipschitz retract of $\mathrm{Map}(W)$.
\end{rem}

\subsection{Finite hull for spheres}
\label{sec:sphere-hulls}

For this section we fix a maximal sphere system $S_0$ in $W$. Let
$\widetilde{S}_0$ be the
full preimage of $S_0$ in $\widetilde{W}$.
We say that a surface in $\widetilde{W}$ has \emph{width $w$
  with respect to $\widetilde{S}_0$}, if it crosses $w$ complementary components of
$\widetilde{S}_0$ in $\widetilde{W}$.
We say that a sphere $s$ in $W$ has
\emph{width $w$ with respect to $S_0$}, if one (hence any) lift
of $s$ to $\widetilde{W}$ has width $w$.
A sphere $\sigma\in \mathcal{S}^0(W)$ has \emph{width $w$}, if its representative $s$ in normal position with respect to $S_0$ has width $w$.
In other words, $s$ has $w$ sphere pieces. This does not depend on the choice of $s$.
Analogously we define the width for the isotopy classes of spheres in $\widetilde{W}$. In fact, the width of the class does not exceed the width of a representative:

\begin{lem}\label{lem:system-normal-pos-width}
  If $\widetilde{s}$ is a sphere in $\widetilde{W}$
  of width $w$, and $\widetilde{s}'$
  is an isotopic sphere in normal position with respect to
  $\widetilde{S}_0$, then $\widetilde{s}'$ has width at most $w$.
\end{lem}
\begin{proof}
  Following the proof of \cite[Prop~1.1]{Ha} we can homotope $\widetilde{s}$ to
  a sphere $\widetilde{s}'$ in normal position by successively removing
  problematic sphere pieces. In each such step one slides the
  problematic piece through a sphere $\widetilde{s}_0\in
  \widetilde{S}_0$. Such a
  homotopy does not increase the width, since
  the side of $\widetilde{s}_0$ that we push $\widetilde{s}$ towards
  already contains a piece of $\widetilde{s}$.
\end{proof}

We are now able to describe the finite hull construction for the sphere
graph. Let $\sigma_1,\ldots,\sigma_k\in \mathcal{S}^0(W)$ be a
collection of spheres with representatives $s_1,\ldots,s_k$. Let $w$
be the maximal width of any $\sigma_i$ with respect to $S_0$. Let
$\mathcal{H}=\mathcal{H}(\sigma_1,\ldots,\sigma_k)$ be
the set of isotopy classes of spheres with width $\leq 2w$ with
respect to $S_0$.
By construction, the set $\mathcal{H}$ is finite and contains all the
$\sigma_i$. The following lemma shows that $\mathcal{H}$ is
$\Pi_\sigma$--convex for all $\sigma=\sigma_i$.
\begin{lem}\label{lem:sphere-normal-pos-width}
  Let $\sigma_1,\ldots,\sigma_n\in \mathcal{S}^0(W)$ be a finite set of spheres. Put $\sigma=\sigma_i$, and let $\rho \in
  \mathcal{H}=\mathcal{H}(\sigma_1,\ldots,\sigma_n)$ be arbitrary. Then each outermost
  surgery pair of $\rho$ in direction of $\sigma$ contains at least
  one element of $\mathcal{H}$.
\end{lem}
\begin{proof}
  Let $s,r$ be representatives of $\sigma,\rho$ which are in normal position
  with respect to $S_0$. We can assume that $s$ and $r$ intersect.
  Let $s_r$ be a representative of $\sigma$ which is in normal
  position with respect to $r$
  Let $r_+,r_-$ be obtained by outermost surgery of $r$ in direction of $s_r$ determined by an outermost disc $D \subset s_r$.

  Let $\widetilde{r}$ be a lift of $r$ to $\widetilde{W}$, and let
  $\widetilde{s}_r$ be the lift of $s_r$ which intersects $\widetilde{r}$
  in the lift of $\partial D$. Denote by $\widetilde{s}$ the lift of $s$ which is
  isotopic to $\widetilde{s}_r$.

  We may homotope $\widetilde{s}$ to be in normal position with respect
  to $\widetilde{r}$ without increasing its width. Namely, let
  $\widetilde{W}(\widetilde{s})$ be the union of all the complementary
  components of $\widetilde{S}_0$ crossed by $\widetilde{s}$.
  The manifold $\widetilde{W}(\widetilde{s})$ is homeomorphic to a
  bordered $3$--sphere (i.e.\ a $3$--sphere minus a finite disjoint
  union of open balls).
  Since $r$ is in normal position with respect to
  $S_0$, the lift $\widetilde{r}$ of $r$ intersects
  $\widetilde{W}(\widetilde{s})$ in a connected surface.
  Therefore, we may homotope $\widetilde{s}$ within
  $\widetilde{W}(\widetilde{s})$ as in the proof of \cite[Prop~1.1]{Ha} so that $\widetilde{s}$ and $\widetilde{r}$ intersect in
  a single circle $\widetilde{c}$.
  The circle $\widetilde{c}$ cuts $\widetilde{r}$ into two discs
  $\widetilde{D}_+,\widetilde{D}_-$.
  The normal position of
  $\widetilde{r}$ and $\widetilde{s}$ is determined by their
  isotopy classes by
  Lemma~\ref{lem:uniqueness-normal-position}. Therefore, there is a
  disc
  $\widetilde{D} \subset \widetilde{s}$ bounded by $\widetilde{c}$, such
  that the spheres $\widetilde{r}_+=\widetilde{D}_+ \cup \widetilde{D} ,\ \widetilde{r}_-=\widetilde{D}_-
  \cup \widetilde{D}$ are isotopic to lifts of $r_+$ and $r_-$.

  Since $\widetilde{r}$ is in normal position with respect to
  $\widetilde{S}_0$, each component of
  $\widetilde{r}-\widetilde{W}(\widetilde{s})$ belongs to exactly one
  of the discs $\widetilde{D}_+,\widetilde{D}_-$. Since $\widetilde{r}$ has width at
  most $2w$, without loss of generality we may assume that the union $\widetilde{D}^*_+$ of all
  the components of $\widetilde{r}-\widetilde{W}(\widetilde{s})$
  contained in $\widetilde{D}_+$ crosses at most $w$ complementary
  components of $\widetilde{S}_0$. The surface $\widetilde{r}_+-\widetilde{D}^*_+$ is contained in $\widetilde{W}(\widetilde{s})$.
  Hence it also has width at most $w$.

  Consequently, the sphere $\widetilde{r}_+$ has width at most $2w$.
  By Lemma~\ref{lem:system-normal-pos-width}, the width of its isotopy class with respect to $\widetilde{S}_0$ is at most $2w$.
  Since $\widetilde{r}_+$ is isotopic to a lift of $r_+$, the isotopy
  class of $r_+$ has width at most $2w$ and the lemma is proved.
\end{proof}

As a consequence of Lemmas~\ref{lem:sphere-projection},
\ref{lem:sphere-normal-pos-width}
and \ref{lem:spheresurgery-equivariance} we obtain
Theorem~\ref{thm:fixedsphere} from Proposition~\ref{prop:fixedpoint}(ii)
in the same way that we obtained Theorem~\ref{thm:fixed_arc}.

\section{Nielsen realisation for graphs}
\label{sec:sph-red}
Let $\{s_1,\ldots,s_k\}$ be a finite set of disjoint spheres in $W$. We say that
$\{s_1,\ldots,s_k\}$ is a \emph{simple sphere system} if the complementary
components of $s_1\cup\dots\cup s_k$ are bordered $3$--spheres or, equivalently, simply-connected.
A clique $\Delta\subset\mathcal{S}^0(W)$ is called \emph{simple} if
some (hence any) set of disjoint representatives of elements in $\Delta$ is a simple sphere system.
In this section we use the results of Section~\ref{sec:spheres} to
show the following.
\begin{thm}\label{thm:fixed-filling-sphere-clique}
  Let $W$ be a (possibly disconnected) doubled handlebody, each component of which has rank $\geq 2$.
  Let $H$ be a finite subgroup of $\mathrm{Map}(W)$. Then $H$ fixes a
  simple clique in the sphere graph $\mathcal{S}(W)$.
\end{thm}
As an immediate consequence of
Theorem~\ref{thm:fixed-filling-sphere-clique} we obtain
Theorem~\ref{thm:graph}. Namely, let $H$ be a finite subgroup of
$\mathrm{Out}(F_n)$. By Theorem~\ref{thm:fixed-filling-sphere-clique}
the group $H$ fixes a simple clique $\Delta$ in $\mathcal{S}(W_n)$ (see the
discussion at the beginning of Section~\ref{sec:spheres}). Then $H$ acts as isometries on the dual
graph of a simple sphere system representing $\Delta$.

\medskip
The proof of Theorem~\ref{thm:fixed-filling-sphere-clique} is similar to the one employed for the
handlebody group in Section~\ref{thm:simple-fixed} --- with
one additional complication:

In the handlebody group case,
if a finite group fixes a clique $\Delta$ in the disc graph, then
using Theorem~\ref{thm:Nielsen}
with hypothesis (B) there is an $H$--invariant hyperbolic metric on the
boundary of the handlebody. We
used curves of a certain type that minimise length with respect to
this metric to extend $\Delta$ to a simple clique.

The manifold $W$ does not have a preferred geometry. We will instead imitate
the length argument using a combinatorial and topological
argument. The role of length
will be played by the geometric intersection number of sphere systems
with arc systems, which we define next.

\subsection{Arc systems}
\label{sec:arcsystems}
We need a slight generalisation of the doubled handlebodies
considered in Section~\ref{sec:spheres}.
Namely, let $W_n^l$ be the $3$--manifold obtained from
$W_n$ by removing $l$ open balls. Let $W_0$ be the
disjoint union of a finite number of such manifolds.

An embedded sphere in $W_0$ is called \emph{essential} if it does not
bound a ball in $W_0$ and if it is not homotopic into a boundary
component of $W_0$.
The definitions of sphere systems and normal position then extend to the
manifold $W_0$ in a straightforward way. Furthermore, normal position
has the same properties as described for $W$ in Section~\ref{sec:spheres}.

An \emph{arc} is a proper embedding of an interval into $W_0$, where the endpoints are allowed to coincide.
Let $S$ be a simple sphere system in $W_0$ and let $a$ be an arc. We call the
intersections of $a$ with complementary components of $S$ the
\emph{arc pieces of $a$ with respect to $S$}. An arc piece $a'\subset a$ is called \emph{returning}, if its endpoints lie on the same boundary component of the complementary
component of $S$ containing $a'$.
We say that an arc $a$ is in \emph{minimal position with respect to
$S$} if no arc piece of $a$ is returning.
\begin{lemma}\label{lem:dual-minimal}
  Let $S,S'$ be two simple sphere systems in $W_0$ which are in
  normal position and let $a$ be an arc. Then $a$ can be
  homotoped to be in minimal position with respect to both $S$
  and $S'$.
\end{lemma}
\begin{proof}
  We replace $a$ with a homotopic arc in minimal position with respect to $S'$ and intersecting $S$ in the minimal possible number of points.
We will prove that $a$ is in minimal position with respect to
  $S$.

  Otherwise, there is a returning arc piece $a' \subset a$ with
  respect to $S$, contained in a complementary component $K$ of
  $S$.
  Let $n$ be the number of sphere pieces of $S'$ in $K$ intersected by $a'$. We
  will argue by induction that $n$ can be decreased to $0$.

  If $n=0$, then $a'$ is disjoint from $S'$. We can then homotope the
  arc piece $a'$ out of $K$ keeping $a$ in minimal position with
  respect to $S'$. Such a homotopy decreases the number of intersection points
  between $a$ and $S$. This contradicts the minimality assumption on $a$.

  If $n>0$, orient $a'$ and let $P'$ be the first sphere piece of $S'$ intersected by $a'$.
  Note that $a'$ intersects $P'$ in a single point, since $P'$
  separates $K$ into two connected components, and $a$ has no
  returning arc piece with respect to $S'$. In particular, the sphere piece $P'$ intersects the boundary
  component of $K$ containing the endpoints of $a'$.
  Therefore we can slide $a'$ along $P'$, keeping $a$ in minimal
  position with respect to $S'$ to decrease the number $n$ of sphere
  pieces of $S'$ in $K$ intersected by $a'$ (see
  Figure~\ref{fig:sliding-arc}), as desired.
  Note that such a homotopy does not increase the
  number of intersection points between $a$ and $S$.
\end{proof}

\begin{figure}[h!]
  \centering
  \includegraphics[width=0.35\textwidth]{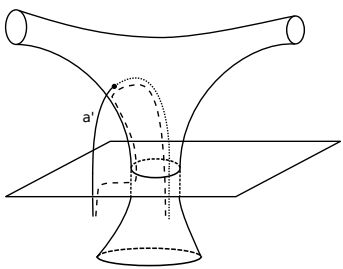}
  \caption{Sliding a returning arc piece to intersect fewer sphere pieces}
  \label{fig:sliding-arc}
\end{figure}

If an arc $a$ is in minimal position with respect to a simple sphere system $S$, then the number of intersection points between $a$ and $S$ depends only on the homotopy classes $\alpha$ of $a$ and $\Sigma$ of $S$. It will be called the \emph{intersection number} $i(\alpha,\Sigma)$.

\medskip

An \emph{arc system} in $W_0$ is a finite collection of arcs in $W_0$.
We say that an arc system $A$ is \emph{generating} if the following
holds for each connected component $\widehat{ W}$ of $W_0$.
\begin{enumerate}[(i)]
\item The graph whose vertices are boundary components of $\widehat{ W}$ and edges correspond to arcs of $A$ in $\widehat{ W}$ is connected.
\item There is a system of loops of $A$ in $\widehat{ W}$ based at a
  point on $\partial\widehat{ W}$, which
  generates the fundamental group of $\widehat{ W}$.
\end{enumerate}

If $A$ and $A'$ are arc systems in the same homotopy class $\Lambda$, then $A$ is generating if and only if $A'$ is generating. In that case we also call $\Lambda$ a \emph{generating arc system}. The \emph{intersection number} $i(\Lambda,\Sigma)$ is the sum of $i(\alpha,\Sigma)$ over all $\alpha\in \Lambda$.

 \begin{lemma}\label{lem:finiteness-small-intersection}
  Let $\Lambda$ be a generating arc system, and let $k>0$. Then there
  are only finitely many isotopy classes $\Sigma$ of simple sphere systems in $W_0$ with intersection number $i(\Lambda,\Sigma)\leq k$.
\end{lemma}
\begin{proof}
  The mapping class group of $W_0$ acts with finite quotient on the
  set of isotopy classes of simple sphere systems.
  Thus there is a finite collection  $\{\Sigma^*_j\}_j$ of isotopy classes of simple
  sphere systems representing all the orbits.
   Let $S$ be a simple sphere
  system with $i(\Lambda,\Sigma)<k$. There is
  is a mapping class $\varphi$ of $W_0$ with
  $\varphi (\Sigma)=\Sigma^*_j$ for some $j$.
  For every $\Sigma^*_j$ there are only finitely many homotopy classes $\alpha$ of arcs
  with intersection number $i(\alpha,\Sigma^*_j)\leq k$.
  Therefore, there are only finitely many possible $\varphi(\Lambda)$. The lemma now follows from the following bound on the number of possible mapping classes $\varphi$.

  \medskip
  \textbf{Claim.} Suppose that $\Lambda$ is a generating arc system. The subgroup of the mapping class group of $W_0$ preserving each oriented arc $\alpha\in \Lambda$ is finite.

  \medskip
  To show the claim, let $f$ be a homeomorphism of $W_0$ such that the arc
  $f(a)$ is homotopic to $a$ for each oriented arc $a \in A$, where $A$ represents $\Lambda$. Let $\widehat{ W}$ be a
  connected component of $W_0$ with $l$ boundary components. If $l>1$, let $s_1,s_2$
  be two boundary components of $\widehat{ W}$ which are connected by an arc
  $a\in A$. These exist by property {(i)} of a generating arc system. Let
  $s$ be the boundary of a regular neighbourhood of
  $s_1\cup a\cup s_2$. Since $f$ preserves the homotopy class of $a$, it also preserves the homotopy classes of
  $s_1$ and $s_2$. Thus the sphere $s$ is homotopic, hence isotopic, to $f(s)$.
  The complement of $s$ in $\widehat{ W}$ is the
  disjoint union of $W_n^{l-1}$ (for a suitable $n\geq 1$) and a
  $3$--sphere minus three open
  balls. The intersection $A\cap W_n^{l-1}$ becomes a generating
  arc system in $W_n^{l-1}$ after a homotopy.
  Arguing inductively using property {(i)} of a generating arc system, we
  find a sphere $\widehat{ s}$ such that the following holds. The complement of
  $\widehat{ s}$ is the disjoint union of
  $W_n^1$ and a $3$--sphere minus $l+1$ open balls. Up to isotopy, the
  homeomorphism $f$ preserves $\widehat{s}$. The restriction
  of $f$ to $W_n^1$ preserves the homotopy class of a generating arc system.
  Now property {(ii)} of a generating arc system implies that $f$ acts as
  the identity on the fundamental group of $W_n^1$. By
  \cite[{Thm~4.3}]{L74} $f$ represents an element of a finite subgroup of the mapping
  class group of $\widehat{W}$.
\end{proof}

\subsection{Fibers}
\label{sec:fibers}
Let $W'$ be the manifold obtained by gluing closed balls to each of
the boundary spheres of $W_0$. There is a natural embedding $W_0 \to W'$.
Let $\Delta'$ be a clique in $\mathcal{S}^0(W')$, and let
$\{s'_1,\ldots,s_k'\}$ be a sphere system in $W'$ representing $\Delta'$. We define
$\mathcal{F}(\Delta')$ to be the following graph. The vertex set of
$\mathcal{F}(\Delta')$ is the set of isotopy classes $\Delta$ of sphere
systems $\{s_1,\ldots,s_k\}$ in $W_0$, such that $s_i$ is isotopic to $s'_i$, viewed
as spheres in $W'$. Two vertices are connected by an edge if the
corresponding sphere systems can be realised in $W_0$ disjointly.

Let $\Lambda$ be a generating arc system in $W_0$ and put
$$I=I(\Lambda,\Delta') = \min\{ i(\Lambda,\Delta) |\ \Delta \in
\mathcal{F}(\Delta')\}.$$
Let $\mathcal{F}^\Lambda(\Delta')$ be the subgraph of
$\mathcal{F}(\Delta')$ induced on the set of the vertices $\Delta$ of
$\mathcal{F}(\Delta')$ satisfying $i(\Lambda,\Delta)=I$.
By Lemma~\ref{lem:finiteness-small-intersection}, the graph $\mathcal{F}^\Lambda(\Delta')$
is finite.

\begin{lemma}
\label{lem:fiber dismantlable}
The graph $\mathcal{F}^\Lambda(\Delta')$ is dismantlable.
\end{lemma}
\begin{proof}
We will show that $\mathcal{F}^\Lambda(\Delta')$ is dismantlable by
defining a dismantling projection.

Let $S_1$ and $S_2$ be two sphere systems
representing vertices of $\mathcal{F}(\Delta')$ which intersect and which
are in normal position. Since $S_1$ and $S_2$ are homotopic in
$W'$, there are discs $D_i \subset S_i$ with common boundary circle, such that $D_1\cup D_2$
bounds a ball in $W'$ disjoint from $S_1\cup S_2$.
We say that such discs $D_i$ are \emph{admissible surgery discs}.

By Lemma~\ref{lem:dual-minimal} we may choose a generating arc system $A$ representing $\Lambda$ in minimal
position with respect to both $S_1$ and $S_2$. The isotopy classes of the sphere systems
$(S_1-D_1)\cup D_2$ and $(S_2-D_2)\cup D_1$ define vertices of $\mathcal{F}(\Delta')$.
By minimality of $I$, the arc system $A$ intersects $D_1$ in the
same number of points as $D_2$.
In particular, the sphere system $(S_1-D_1)\cup D_2$ represents a
vertex of $\mathcal{F}^\Lambda(\Delta')$.

We call $(S_1-D_1)\cup D_2$ an \emph{admissible surgery of $S_1$ in
  direction of $S_2$}.
Since normal position is unique, the isotopy classes of admissible
surgeries only depend on the isotopy classes $\Sigma_1,\Sigma_2$ of
the sphere systems $S_1$ and $S_2$.
Let $\Pi_{\Sigma_2}(\Sigma_1)$ be the set of all
such admissible surgeries (each of which is interpreted as a pair
whose two elements coincide).

Lemma
~\ref{lem:no-sphere-cycles} holds for $W_0$ with the same proof as
for $W$. Hence, arguing exactly as in the proof of
Lemma~\ref{lem:sphere-projection}, one
shows that $\Pi_{\Sigma_2}$ is a dismantling projection for
each $\Sigma_2\in\mathcal{F}^\Lambda(\Delta')$.
\end{proof}

\begin{proof}[Proof of Theorem~\ref{thm:fixed-filling-sphere-clique}]
  We prove the theorem by induction on the sum of the ranks of the components of $W$ minus the number of components of $W$.
  By Theorem~\ref{thm:fixedsphere}, the subgroup $H$ fixes a
  clique $\Delta$ in $\mathcal{S}(W)$.
  Let $S=\{s_1,\ldots,s_n\}$ be
  a sphere system representing $\Delta$.
  If every complementary component of $S$ is a bordered $3$--sphere, then
  $\Delta$ is already simple.
  Otherwise, let $W_0$ be the union of all the complementary components of
  $S$ which are not bordered $3$--spheres. The finite group $H$ acts as a group of mapping classes on $W_0$.

  Let $W'$ be the manifold obtained by gluing closed balls to each
  boundary component of $W_0$.
  There is an $H$--invariant simple clique $\Delta'\subset \mathcal{S}^0(W')$:
  in each rank $1$ component of $W'$ we consider the unique homotopy class of
  spheres; in the union of the remaining components of $W'$ there is a simple sphere system which is $H$--invariant up to isotopy by
  the induction hypothesis.

  Let $\Lambda$ be a generating arc system in $W_0$ which is fixed by $H$. Such a system can be obtained as the orbit of
  an arbitrary generating arc system.
  Since $H$ preserves both $\Lambda$ and $\Delta'$, it acts by automorphisms on
  $\mathcal{F}^\Lambda(\Delta')$. By Lemma~\ref{lem:fiber dismantlable} the graph $\mathcal{F}^\Lambda(\Delta')$ is dismantlable.
  Therefore, by Proposition~\ref{prop:fixedpoint}(i) the group $H$ fixes a clique $\Delta^*$ in
  $\mathcal{F}^\Lambda(\Delta')$. Since $\Delta'$ is a simple clique, the vertices of $\Delta^*$ can be realised as a simple sphere system in $W_0$.
  We can interpret $\Delta^*$ as a clique in $\mathcal{S}^0(W)$ which together with
  $\Delta$ forms the desired simple clique fixed by $H$.
\end{proof}

\begin{rem}
  In fact, one can show that if $W'$ does not have rank $1$ components, then the flag complex spanned on the graph $\mathcal{F}(\Delta')$ is
  isomorphic to a triangulation of a product of trees with orthoschemes. The
  number of factors is the number of boundary components
  of $W_0$. In the case where $W_0$ has one boundary component, this means
  means that $\mathcal{F}(\Delta')$ is a tree. That case can be proved for example by extending the dismantling projection from $\mathcal{F}^\Lambda (\Delta')$ (see Lemma~\ref{lem:fiber dismantlable}) to the whole $\mathcal{F}(\Delta')$. This fact is
  analogous to a result of Kent--Leininger--Schleimer \cite[Thm~7.1]{KLS}
  for the curve graph of a surface with one puncture. In the case where $W_0$ has more boundary components, recognising the product of trees structure requires local analysis of $\mathcal{F}(\Delta')$.
\end{rem}

\section{Contractibility of the set of fixed points}
\label{sec:contr}
This section is a continuation of Section~\ref{sec:dism}. The main result is the combinatorial Theorem~\ref{thm:contrfix} stating that in a finite flag simplicial complex with dismantlable $1$--skeleton the set of points fixed by a finite automorphism group is contractible. It is a significant extension of Theorem~\ref{thm:fixedpoint}, and according to our knowledge it is a new result. Theorem~\ref{thm:contrfix} is of independent interest in graph theory, and as we will see it has numerous applications in geometric group theory. We will use it to obtain Theorem~\ref{thm:EGhyper} in Section~\ref{sec:Rips}. At the end of the current section we will deduce from Theorem~\ref{thm:contrfix} a criterion for contractibility of fixed point sets in the presence of dismantling projections (Proposition~\ref{prop:contr}). This will allow us to prove Metatheorems~\ref{met:H_infinite} and~\ref{met:H_filling}.

We adopt a convention that in a flag simplicial complex a \emph{neighbourhood} of a vertex is its neighbourhood in the $1$--skeleton. Similarly a vertex is \emph{dominated} (resp.\ \emph{dominating}) if is dominated (resp.\ dominating) in the $1$--skeleton etc.

\begin{proof}[Proof of Theorem~\ref{thm:contrfix}]
  First observe that $Y^H$ is nonempty, by Theorem \ref{thm:fixedpoint}.
  The proof of the theorem is by induction on the number $|V|$ of vertices of $Y$. If $|V|=1$ then the assertion is clear. Now, for a given $Y$ assume that the theorem has been proved for all complexes with less than $|V|$ vertices.
  We treat separately two cases (as in the proof of Theorem \ref{thm:fixedpoint}).
  \medskip

  \noindent
  \emph{Case 1.} There exist no two vertices with common neighbourhood. In this case let $V'\subset V$ be the nonempty set of dominated vertices which are not dominating. Note that the subcomplex $Y'$ spanned on $V\setminus V'$ is $H$--invariant, dismantlable (by Lemma \ref{lem:pushing}), with less than $|V|$ vertices. By induction hypothesis $Y'^H$ is contractible.

  For every vertex $\sigma \in V'$ let $\sigma^{\mathrm{dom}}$ denote the simplex of $Y$ spanned on the set of all vertices dominating $\sigma$. Observe that $\sigma^{\mathrm{dom}}$ is contained in $Y'$.
  For $\sigma\in V\setminus V'$, put $\sigma^{\mathrm{dom}}=\sigma$. By $b_{\sigma^{\mathrm{dom}}}$ we denote the barycentre of ${\sigma^{\mathrm{dom}}}$.
    On the (geometric realisation of the) simplicial complex $Y$ define a deformation retraction $\mathcal{R}\colon Y\times [0,1] \to Y$ as follows. First we define $\mathcal{R}$ on vertices of $Y$ by setting $\mathcal{R}(\sigma,t)=(1-t)\sigma + t b_{\sigma^{\mathrm{dom}}}$, where the linear combination is taken with respect to the standard affine structure on $Y$.
    If the set $S=\{ \sigma_1,\ldots,\sigma_k\}\subset V$ of vertices spans a simplex of $Y$, then there exists a simplex $\Delta$ of $Y$ containing $S$ and all the simplices $\sigma_i^{\mathrm{dom}}$. Thus, for every $t$ the set $\{\mathcal{R}( \sigma_1,t),\ldots,\mathcal{R}(\sigma_k,t)\}$
is contained in $\Delta$. It follows that we may extend $\mathcal{R}$ affinely to a map on $Y\times [0,1]$. Observe that $\mathcal{R}( Y\times \{1 \})=Y'$.

Since for every vertex $\sigma \in V$ we have $H(\sigma^{\mathrm{dom}})=(H\sigma)^{\mathrm{dom}}$, the map $\mathcal{R}$ is $H$--equivariant (with the trivial $H$--action on $[0,1]$).
Thus, if the set $\{ \sigma_1,\ldots,\sigma_k\}$ of vertices of $Y$ spans an $H$--invariant simplex, then the set $\{\mathcal{R}( \sigma_1,t),\ldots,\mathcal{R}(\sigma_k,t)\}$ is $H$--invariant for every $t$.
  It follows that the image $\mathcal{R}(b,t)$ of the barycentre $b$ of an $H$--invariant simplex of $Y$ is a point fixed by $H$. Moreover, every point $x\in Y$ fixed by $H$ is an affine combination of such barycentres, thus $\mathcal{R}(x,t)$ is fixed by $H$.
   Thus the image $\mathcal{R}(Y^H,t)$ of the set of the points fixed by $H$ consists itself of fixed points. We obtain a deformation retraction $\mathcal{R}|_{Y^H\times [0,1]}$ of $Y^H$ to $\mathcal{R}( Y^H\times \{1 \})$.
   Since $\mathcal{R}( Y^H\times \{1 \})$ is isomorphic to $Y'^H$ which is contractible, we conclude that $Y^H$ is contractible.
\medskip

\noindent
\emph{Case 2.}
   There exist different vertices of $Y$ having common neighbourhood. In that
case consider the graph $\Gamma'$ obtained by identifying all vertices with
common neighbourhood (as in the proof of Theorem~\ref{thm:fixedpoint}). Let $Y'$ denote the corresponding flag simplicial complex,
i.e. $Y'=\Gamma'^\blacktriangle$. The $H$--action on $Y$ induces an $H$--action
on $Y'$, and $\Gamma'$ is dismantlable by Lemma \ref{lem:pushing} (see the proof of Theorem~\ref{thm:fixedpoint}). Thus, by
induction hypothesis $Y'^H$ is contractible.

   Observe that for every vertex $\sigma$ of $\Gamma$ the set of vertices with
the same neighbourhood as $\sigma$ spans a simplex $\Delta^{\sigma}$ of $Y$.
   Moreover, two such simplices $\Delta^{\sigma}\neq \Delta^{\rho}$ are
disjoint.
   Define a deformation retraction $\mathcal{R}\colon Y \times [0,1] \to Y$ as
follows.
   First we define $\mathcal{R}$ on vertices of $Y$ by setting
$\mathcal{R}(\sigma, t)=(1-t)\sigma + t b_{\Delta^{\sigma}}$.
   Note that if vertices $\sigma_1,\ldots,\sigma_k$ of $Y$ are pairwise connected by
edges (i.e. when they span a simplex), then for every $t$ all the images
$\mathcal{R}(\sigma_1,t),\ldots,\mathcal{R}(\sigma_k,t)$ are contained in a common
simplex of $Y$. Hence $\mathcal{R}$ can be affinely extended to a map
$\mathcal{R}\colon Y \times [0,1] \to Y$.
Let $\mathcal T$ be the flag triangulation of $\mathcal{R}( Y\times \{1 \})$
defined as follows. Vertices of $\mathcal T$ are the points
$b_{\Delta^{\sigma}}$, for all vertices $\sigma$ of $Y$. Edges are straight
segments connecting $b_{\Delta^{\sigma}}$ and $b_{\Delta^{\rho}}$, whenever
$\sigma$ and $\rho$ are connected by an edge of $Y$. Simplices are convex hulls
of cliques.
Observe that $\mathcal T$ is isomorphic to $Y'$.

   Similarly as in the previous case we see that the map $\mathcal{R}$ is
$H$--invariant. Thus, again, if the set $\{ \sigma_1,\ldots,\sigma_k\}$ of vertices
of $Y$ spans an $H$--invariant simplex, then for every $t$ the set
$\{\mathcal{R}( \sigma_1,t),\ldots,\mathcal{R}(\sigma_k,t)\}$ is $H$--invariant.
   Analogously to the previous case we conclude that we obtain a deformation
retraction $\mathcal{R}|_{Y^H\times [0,1]}$ of $Y^H$ to $\mathcal{R}( Y^H\times
\{1 \})$. Observe that the latter space is isomorphic (as a subcomplex of the barycentric subdivision of the
triangulation $\mathcal T$) to $Y'^H$, which is contractible.
\end{proof}

\begin{defin}
Assume that for every vertex $\sigma$ of a graph $\Gamma$ with vertex set $V$ we have a $\sigma$--projection $\Pi_\sigma$. The family $\{\Pi_\sigma\}_\sigma$ is \emph{synchronised} if
for any clique $\Delta\subset V$ and any sequence $\rho_1,\rho_2,\ldots$ with $\rho_{i+1}\in \Pi^*_{\sigma_i}(\rho_i)$ with any $\sigma_i\in \Delta$, the sequence $(\rho_i)_i$ enters (and then stays in) $\Delta$.
\end{defin}

\begin{prop}
\label{prop:contr}
Let $H$ be a group of automorphisms of a graph $\Gamma$ with vertex set $V$ without infinite cliques. Assume that we have a $\sigma$--projection $\Pi_\sigma$ for each $\sigma\in V$ and the family $\{\Pi_\sigma\}_\sigma$ is $H$--equivariant. Suppose that
\begin{enumerate}[(i)]
\item
$H$ fixes a vertex of $\Gamma$, or
\item
$H$ fixes a clique in $\Gamma$ and the family $\{\Pi_\sigma\}_\sigma$ is synchronised.
\end{enumerate}
Then the fixed point set ${\Gamma^\blacktriangle}^H$ is contractible.
\end{prop}

\begin{proof}
We first prove the proposition under hypothesis (ii). Recall that for $S\subset V$, the subcomplex of $\Gamma^\blacktriangle$ spanned on $S$ is denoted by $S^\blacktriangle$. In particular, $\Gamma^\blacktriangle=V^\blacktriangle$.

Let $S\subset V$ be the union of the vertices of all $H$--invariant cliques. We can exhaust $S$ by an increasing sequence $S_1,S_2,\ldots$ of finite $H$--invariant sets. Choose an $H$--invariant clique $\Delta\subset S$.

Let $R_n\subset V$ be the minimal set containing $S_n$ and with the property that for any $\rho\in R_n$ and $\sigma\in \Delta$, all the elements of $\Pi^*_\sigma(\rho)$ lie in $R_n$ as well. Since $S_n$ is finite, the set $R_n$ is finite by the fact that $\{\Pi_\sigma\}_\sigma$ is synchronised and by K\"onig's lemma. Since both $\Delta$ and $S_n$ are $H$--invariant and $\{\Pi_\sigma\}_\sigma$ is $H$--equivariant, the set $R_n$ is $H$--invariant as well.

Hence $R_n$ are finite $H$--invariant sets that contain $S_n$ and are $\Pi_\sigma$--convex for each $\sigma\in \Delta$. The subcomplexes
$R^\blacktriangle_n$ exhaust
$S^\blacktriangle$ which contains ${V^\blacktriangle}^H$.
By Corollary~\ref{cor:proj->dism} the subgraph induced on each $R_n$ is dismantlable.
By Theorem~\ref{thm:contrfix} each ${R_n^\blacktriangle}^H$ is contractible and hence ${V^\blacktriangle}^H$ is contractible, as desired.

Under hypothesis (i) we can take $\Delta$ to be a vertex and we can use axiom (ii) of a dismantling projection to obtain that $R_n$ is finite. The remaining part of the argument is the same.
\end{proof}

\section{Contractibility of fixed point sets in the arc complex}
\label{sec:contr_arc}

In this section we prove Metatheorems~\ref{met:H_infinite} and~\ref{met:H_filling} for the mapping class group (Theorems~\ref{thm:contrarc} and ~\ref{thm:arc_classify}). The \emph{arc complex} $\mathcal{A}^\blacktriangle(X)$ is the flag simplicial complex obtained from the arc graph $\mathcal{A}(X)$ by spanning simplices on all of its cliques. Then simplices of $\mathcal{A}^\blacktriangle(X)$ correspond to sets of disjoint arcs.

\begin{thm}
\label{thm:contrarc}
Let $X$ be a closed connected oriented surface with nonempty
set of marked points and negative Euler characteristic.
Let $H$ be any subgroup of $\mathrm{Map}(X)$. Then the set $\mathcal A^{\blacktriangle}(X)^H$ of points of the arc complex $\mathcal A^{\blacktriangle}(X)$ fixed by $H$ is empty or contractible.
\end{thm}

Hence in view of Theorem~\ref{thm:fixed_arc} in the particular case where $H$ is finite the fixed point set $\mathcal A^{\blacktriangle}(X)^H$ is contractible. Also, if we take for $H$ the trivial group, it follows that $\mathcal A^{\blacktriangle}(X)$ is contractible, which was proved in \cite{Har1, Ha0}.

\begin{lemma}
\label{lem:synchr_arc}
The family $\{\Pi_\sigma\}_\sigma$ of dismantling projections on $\mathcal A(X)$ from Section~\ref{sec:fixarc} is synchronised.
\end{lemma}
\begin{proof}
Choose a clique $\Delta\subset \mathcal A^0(X)$ and any sequence $\rho_1,\rho_2,\ldots$ with $\rho_{i+1}\in \Pi^*_{\sigma_i}(\rho_i)$ with some $\sigma_i\in \Delta$. Then, as in the proof of Lemma~\ref{lem:arcprojection}(ii), the geometric intersection number between $\rho_i$ and the set $\Delta$ of disjoint arcs decreases until $\rho_i\in \Delta$, as desired.
\end{proof}

\begin{proof}[Proof of Theorem~\ref{thm:contrarc}]
Assume that $\mathcal A^{\blacktriangle}(X)^H$ is nonempty, i.e.\ $H$ fixes a clique in $\mathcal A(X)$. Cliques in $\mathcal A(X)$ are finite. The family $\{\Pi_\sigma\}_\sigma$ of dismantling projections on $\mathcal A(X)$ from Lemma~\ref{lem:arcprojection} is $H$--equivariant (Lemma~\ref{lem:arcsurgery-equivariance}) and synchronised (Lemma~\ref{lem:synchr_arc}). Hence by Proposition~\ref{prop:contr}(ii), the fixed point set $\mathcal A^{\blacktriangle}(X)^H$ is contractible.
\end{proof}

Denote by $\mathcal{A}^{0}_\mathrm{fill}(X)$ the barycentres of simplices of $\mathcal{A}^\blacktriangle(X)$ spanned on filling cliques. Let $\mathcal{A}^\blacktriangle_\mathrm{fill}(X)$ be the subcomplex of the barycentric subdivision of $\mathcal{A}^\blacktriangle(X)$ spanned on $\mathcal{A}^{0}_\mathrm{fill}(X)$. We call $\mathcal{A}^\blacktriangle_\mathrm{fill}(X)$ the \emph{filling arc system complex}.
The stabiliser in $\mathrm{Map}(X)$ of a vertex (or a simplex) of $\mathcal{A}^\blacktriangle_\mathrm{fill}(X)$ is finite. Conversely, we have proved in Theorem~\ref{thm:fixed_arc_system} that for each finite subgroup $H\subset\mathrm{Map}(X)$ the fixed point set $\mathcal{A}^\blacktriangle_\mathrm{fill}(X)^H$ contains a vertex. We extend this to the following.

\begin{thm}
\label{thm:arc_classify}
Let $X$ be a closed connected oriented surface with nonempty set of marked points and negative Euler characteristic. Let $H$ be a finite subgroup of $ \mathrm{Map}(X)$. Then the set $\mathcal{A}^\blacktriangle_\mathrm{fill}(X)^H$ of points in the filling arc system complex $\mathcal{A}^\blacktriangle_\mathrm{fill}(X)$ fixed by $H$ is contractible.
\end{thm}

Before we give the proof of Theorem~\ref{thm:arc_classify}, we note that consequently we obtain the following.

\begin{cor}
\label{cor:spine_teich}
Let $X$ be a closed connected oriented surface with nonempty set of marked points and negative Euler characteristic. The filling arc system complex $\mathcal{A}^\blacktriangle_\mathrm{fill}(X)$ is a finite model for $\underline{E}\mathrm{Map}(X)$ --- the classifying space for proper actions for $\mathrm{Map}(X)$.
\end{cor}

\emph{Harer's spine} of the Teichm\"uller space of $X$ (see \cite{Har2}) is the subcomplex of $\mathcal{A}^\blacktriangle_\mathrm{fill}(X)$ spanned on the vertices involving arcs starting and terminating on a distinguished marked point. Hence in the case of one marked point $\mathcal{A}^\blacktriangle_\mathrm{fill}(X)$ coincides with Harer's spine. The proof of Theorem~\ref{thm:arc_classify} carries over to the action of the pure mapping class group of $X$ on Harer's spine.

\begin{cor}
\label{cor:spine_teich_Harer}
Let $X$ be a closed connected oriented surface with nonempty set of marked points and negative Euler characteristic. Then Harer's spine of Teichm\"uller space of $X$ is a finite model for the classifying space for proper actions for the pure mapping class group of $X$.
\end{cor}

The dimension of Harer's spine coincides with the virtual cohomological dimension of $\mathrm{Map}(X)$, so it is efficient in this sense.

\medskip

The proof of Theorem~\ref{thm:arc_classify} relies on defining a surgery procedure for filling arc sets. In order to preserve the property of being filling we will now need to include in the surgered set both arcs obtained by outermost surgery. This will give rise to a dismantling projection on the $1$--skeleton of the filling arc system complex.

\begin{proof}[Proof of Theorem~\ref{thm:arc_classify}]
Let $\Delta\in{\mathcal{A}^{0}_\mathrm{fill}(X)}^H$ which is nonempty by Theorem~\ref{thm:fixed_arc_system}.
We construct a $\Delta$--projection $\Pi_\Delta$ on the $1$--skeleton $\mathcal{A}_\mathrm{fill}(X)$ of $\mathcal{A}^\blacktriangle_\mathrm{fill}(X)$.
Let $\Sigma\in\mathcal{A}^{0}_\mathrm{fill}(X)$ be any vertex distinct from $\Delta$. Every pair $\{\Pi_1,\Pi_2\}\in\Pi_\Delta(\Sigma)$ will satisfy $\Pi_1=\Pi_2$, so $\Pi_\Delta(\Sigma)$ will be determined by $\Pi^*_\Delta(\Sigma)$.

Note that $\Delta, \Sigma$ can be interpreted as filling sets of (homotopy classes of) disjoint arcs. If all the arcs in $\Delta, \Sigma$ are disjoint, we put $\Pi^*_\Delta(\Sigma)=\{\Delta\}$. Otherwise, realise the homotopy classes of arcs $\Delta, \Sigma$ as families of genuine arcs $A,S\subset X$ in minimal position. Choose an arc $a\in A$ intersecting $S$. Let $b\subset a$ be a component of $a-S$ sharing an endpoint with $a$.
Let $s\in S$ be the arc containing the endpoint of $b$ in the interior of $a$. Consider $s_+,s_-$ obtained by outermost surgery of $s$ in direction of $a$ determined by $b$. Let $P=S\cup\{s_+,s_-\}\setminus \{s\}$. Note that $P$ is filling, since any simple closed curve that is disjoint from both $s_+,s_-$ up to homotopy is also disjoint from $s$ up to homotopy.
Define $\Pi^*_\Delta(\Sigma)$ to be the set of homotopy classes of all possible $P$, under all choices of $a$ and $b$. This does not depend on the realisations $A,S$ of $\Delta, \Sigma$.

The assignment $\Pi_\Delta(\Sigma)$ determined by $\Pi^*_\Delta(\Sigma)$ satisfies the axioms of a $\Delta$--projection by the same argument as in the proof of Lemma~\ref{lem:arcprojection}:

For axiom (i) we consider the surface $\check{X}$ obtained from $X$ by
replacing each marked point with a boundary component and we put a
hyperbolic metric with geodesic boundary on $\check{X}$. We realise
all the arcs of $\Delta$ and the arbitrary set $R=\{\Sigma_1,\ldots,
\Sigma_k\}\subset \mathcal{A}^{0}_\mathrm{fill}(X)$ as families $A,
S_1,\ldots, S_k$ of shortest geodesic arcs on $\check{X}$. Choose any arc $a\in A$ intersecting $\bigcup^k_{i=0} S_i$.
Let $b$ be a component of $a-\bigcup_{i=0}^k S_i$ sharing an endpoint with $a$. Choose $S_i$ so that it contains the endpoint of $b$ in the interior of $a$. Let $\Pi\in\Pi^*_\Delta(\Sigma_i)$ be obtained by surgery determined by $b$. Then $N(\Pi)\supset N(\Sigma_i)\cap R$, as desired.

For axiom (ii) we observe that the sum of the geometric intersection numbers between all of the arcs of $\Delta$ and all of the arcs of any element of $\Pi^*_\Delta(\Sigma)$ is less than the corresponding sum for $\Delta$ and $\Sigma$.

The map $\Pi_\Delta$ is $H$--equivariant and cliques in $\mathcal{A}_\mathrm{fill}(X)$ are finite.
Hence by Proposition~\ref{prop:contr}(i) the fixed point set $\mathcal{A}^\blacktriangle_\mathrm{fill}(X)^H$ is contractible, as desired.

\end{proof}

\section{Contractibility of fixed point sets in the disc and sphere complexes}
\label{sec:contr_disc_sphere}

In this section we prove Metatheorems~\ref{met:H_infinite} and~\ref{met:H_filling} for the handlebody group and $\mathrm{Out}(F_n)$.

\subsection{Disc complex}

The \emph{disc complex} $\mathcal D^{\blacktriangle}(U)$ is the flag simplicial complex obtained from the disc graph $\mathcal D(U)$ by spanning simplices on all of its cliques. Then simplices of $\mathcal{D}^\blacktriangle(U)$ correspond to sets of disjoint discbounding curves.

\begin{thm}
\label{thm:contrdisc}
 Let $U$ be a connected handlebody of genus $\geq 2$. Let $H$ be any subgroup of $\mathrm{Map}(U)$. Then the set $\mathcal D^{\blacktriangle}(U)^H$ of points of the disc complex $\mathcal D^{\blacktriangle}(U)$ fixed by $H$ is empty or contractible.
\end{thm}

\begin{proof}
Cliques in $\mathcal D(U)$ are finite. The family of dismantling projections for $\mathcal D(U)$ from Lemma~\ref{lem:discsurgery} is $H$--equivariant by Lemma~\ref{lem:discsurgery-equivariance}. It is also synchronised which is proved exactly as Lemma~\ref{lem:synchr_arc}. By Proposition~\ref{prop:contr}(ii), if the set $\mathcal D^{\blacktriangle}(U)^H$ is nonempty, then it is contractible.
\end{proof}

Let $\mathcal{D}_{\mathrm{simp}}^0(U)$ be the set of barycentres of simplices in $\mathcal D^{\blacktriangle}(U)$ spanned on simple cliques.
The \emph{simple disc system complex} $\mathcal{D}^{\blacktriangle}_{\mathrm{simp}}(U)$ is the subcomplex of the barycentric subdivision of  $\mathcal D^{\blacktriangle}(U)$ spanned on $\mathcal{D}^0_{\mathrm{simp}}(U)$.

\begin{thm}
\label{thm:handle_spine}
 Let $U$ be a connected handlebody of genus $\geq 2$. Let $H$ be a finite subgroup of $\mathrm{Map}(U)$. Then the set $\mathcal{D}^{\blacktriangle}_\mathrm{simp}(U)^H$ of points of the simple disc system complex $\mathcal{D}^{\blacktriangle}_\mathrm{simp}(U)$ fixed by $H$ is contractible.
\end{thm}

Note that the simple disc system complex $\mathcal{D}^\blacktriangle_\mathrm{simp}(U)$ is not a finite model for $\underline{E}\mathrm{Map}(U)$, since stabilisers of simple disc systems are not all finite (they extend to braid groups).

\begin{proof}
By Theorem~\ref{thm:fixed_disc_system}, the fixed point set $\mathcal D^{\blacktriangle}_\mathrm{simp}(U)^H$ contains a vertex $\Delta\in \mathcal D^0_\mathrm{simp}(U)$.

We define a $\Delta$--projection $\Pi_\Delta$ on the $1$--skeleton of $\mathcal{D}^{\blacktriangle}_\mathrm{simp}(U)$. Let $\Sigma\in D^0_\mathrm{simp}(U)\setminus \{\Delta\}$. If all the discbounding curves of $\Delta$ and $\Sigma$ are disjoint up to homotopy, we put $\Pi^*_\Delta(\Sigma)=\{\Delta\}$. Otherwise, we realise $\Delta, \Sigma$ as families of disjoint discbounding curves $D,S\subset \partial U$ pairwise in minimal position. Consider a subarc $a\subset d$ for some $d\in D$ which is outermost in a disc bounded by $d$ for a disc bounded by some $s\in S$. Assume that $a$ does not contain any other such outermost subarc. For any such subarc $a$ we consider $P=S\cup\{s_+,s_-\}\setminus \{s\}$, where $s_+,s_-$ are obtained by outermost surgery of $s$ in direction of $d$ determined by $a$. The set $P$ is a simple disc system, and we define $\Pi_\Delta^*(\Sigma)$ to be the set of all homotopy classes of all $P$ constructed using all such subarcs $a$. This does not depend on the realisations $D,S$ of $\Delta,\Sigma$. Let $\Pi_\Delta(\Sigma)$ be the family of pairs $\{\Pi_1,\Pi_2\}$, where $\Pi_1=\Pi_2\in \Pi_\Delta^*(\Sigma)$.

We leave it to the reader to verify that $\Pi_\Delta$ satisfies the axioms of a dismantling projection and that it is $H$--equivariant. This can be carried out exactly as in the proof of Theorem~\ref{thm:arc_classify}, building upon the proof of Lemma~\ref{lem:discsurgery}.
Since cliques in $\mathcal D^{\blacktriangle}_\mathrm{simp}(U)$ are finite, by Proposition~\ref{prop:contr}(i) the fixed point set $\mathcal D^{\blacktriangle}_\mathrm{simp}(U)^H$ is contractible.
\end{proof}

\subsection{Sphere complex}

The \emph{sphere complex} $\mathcal S^{\blacktriangle}(W)$ is the flag simplicial complex obtained from the sphere graph $\mathcal S(W)$ by spanning simplices on all of its cliques. Then simplices of $\mathcal{S}^\blacktriangle(W)$ correspond to sets of disjoint spheres.

\begin{thm}
\label{thm:contrsphere}
 Let $W$ be a doubled connected handlebody of rank $n\geq 2$. Let $H$ be any subgroup of $\mathrm{Map}(W)$ or $\mathrm{Out}(F_n)$. Then the set $\mathcal S^{\blacktriangle}(W)^H$ of points of the sphere complex $\mathcal S^{\blacktriangle}(W)$ fixed by $H$ is empty or contractible.
\end{thm}

We omit the proof since it is carried out exactly as the proofs of Theorems~\ref{thm:contrarc} and~\ref{thm:contrdisc}.

\medskip

Let $\mathcal{S}_{\mathrm{simp}}^0(W)$ be the set of barycentres of simplices in $\mathcal S^{\blacktriangle}(W)$ spanned on simple cliques.
The \emph{simple sphere system complex} $\mathcal S^{\blacktriangle}_{\mathrm{simp}}(W)$ is the subcomplex of the barycentric subdivision of  $\mathcal S^{\blacktriangle}(W)$ spanned on $\mathcal{S}^0_{\mathrm{simp}}(W)$. The simple sphere system complex is the spine of the Culler--Vogtmann Outer space \cite{Ha}. Our method gives a new proof of the following.

\begin{thm}[\cite{Whi}, \cite{KV}]
\label{thm:sphere_spine}
 Let $W$ be a doubled connected handlebody of genus $n\geq 2$. Let $H$ be a finite subgroup of $\mathrm{Map}(W)$ or $\mathrm{Out}(F_n)$. Then the set $\mathcal S^{\blacktriangle}_\mathrm{simp}(W)^H$ of points of the simple sphere system complex $\mathcal S^{\blacktriangle}_\mathrm{simp}(W)$ fixed by $H$ is contractible.
\end{thm}

\begin{cor}[\cite{Whi}, \cite{KV}]
\label{cor:spine_sphere}
Let $W$ be a doubled connected handlebody of genus $n\geq 2$. The simple sphere system complex $\mathcal{S}^\blacktriangle_\mathrm{simp}(W)$ is a finite model for $\underline{E}\mathrm{Map}(W)$ and $\underline{E}\mathrm{Out}(F_n)$.
\end{cor}

We leave the proof of Theorem~\ref{thm:sphere_spine} to the reader. It requires defining a dismantling projection on the $1$--skeleton of $\mathcal S^{\blacktriangle}_\mathrm{simp}(W)$ in the same way as it was done in Theorems~\ref{thm:arc_classify} and~\ref{thm:handle_spine} for the filling arc system complex and the simple disc system complex. To verify that it satisfies the axioms of a dismantling projection one needs to invoke the proof of Lemma~\ref{lem:sphere-projection}.

\section{Rips complex}
\label{sec:Rips}
In this section we prove Theorem~\ref{thm:EGhyper}, that if a group $G$ is $\delta$--hyperbolic and $D\geq 8\delta +1$, then the Rips complex $P_D(G)$ is a finite model for $\underline EG$.

We say that a graph is \emph{$\delta$--hyperbolic}, with $\delta\in \N$, if for any vertices $u,v,w$, for any geodesics $uv,vw,wu$, and for any vertex $t$ on $uv$ there is a vertex on $vw\cup wu$ at distance $\leq \delta$ from $t$.
 A finitely generated group $G$ is \emph{$\delta$--hyperbolic} if its Cayley graph with respect to some finite generating set is $\delta$--hyperbolic. For $D\in \N$, the \emph{Rips complex} $P_D(G)$ is the flag simplicial complex with vertex set $G$, and edges connecting two vertices at distance at most $D$ in the Cayley graph.

Theorem~\ref{thm:EGhyper} is an immediate consequence of Theorem \ref{thm:contrfix} and Lemma \ref{lem:dismhyp} below.

\begin{lem}
\label{lem:dismhyp}
   Let $G$ be a $\delta$--hyperbolic group and let $D\geq 8\delta +1$. Let $H$ be a finite subgroup of $G$. For every finite subset $S$ of $G$ there exists an $H$--invariant finite dismantlable subgraph of the $1$--skeleton of $P_D(G)$ containing $S$.
\end{lem}

\begin{proof}
If $\delta=0$, then the Cayley graph of $G$ is a tree ($G$ is a free group) and the lemma follows from the fact that for any finite tree the graph obtained after connecting by edges vertices at distance $\leq D$ is dismantlable. We can now assume $\delta\geq 1$.

The \emph{ball} $B_R(C)$ of radius $R\in \N$ around a subset $C$ of the vertex set $G$ of the Cayley graph is the set all of vertices at distance at most $R$ from some vertex in $C$. For a subset $O\subset G$ let $R$ be minimal such that $O$ is contained in $B_R(v)$ for some $v\in G$.
The \emph{quasi-centre} of $O$ is the set of all $v\in G$ with $O\subset B_R(v)$. By \cite[Lem III.$\Gamma$.3.3]{BH} the diameter of the quasi-centre of any finite subset of $G$ is at most $4\delta+1$.

Let $C\subset G$ be the quasi-centre of the orbit $HS$.
For any $r\in \N$ the ball $B_r(C)$ is $H$--invariant and for $r$ large enough we have $S\subset B_r(C)$.
Thus to prove the lemma it suffices to show that for any $r\in \N$ the graph induced on $B_r(C)$ in the $1$--skeleton of $P_D(G)$ is dismantlable.
If $r< 2\delta$, then $B_r(C)$ spans a simplex in $P_D(G)$ and the assertion is clear. Thus it remains to prove the following claim. By $d(\cdot,\cdot)$ we denote the distance in the Cayley graph.

\begin{claim} Consider a vertex $v\in G$ at distance $a\geq 2\delta$ from $C$. Let $w\in C$ be at distance $a$ from $v$. Let $u\in G$ be a vertex lying on a geodesic $vw$ with $d(u,v)=2\delta$. Then $u$ dominates $v$ in the subgraph of the $1$--skeleton of $P_D(G)$ induced on $B_a(C)$.
\end{claim}

\noindent \emph{Proof of Claim.}
The proof is similar to the proof of \cite[Prop III.$\Gamma$.3.23]{BH}.
Let $t\in B_a(C)$ be such that $d(t,v)\leq D$, i.e.\ $t$ is a neighbour of $v$ in $P_D(G)$. We have to show that $d(t,u)\leq D$.
Consider geodesics $wt$ and $tv$. By $\delta$--hyperbolicity we have $d(u,u')\leq \delta$ for some vertex $u'\in wt  \cup  tv$.
There are two cases:
\medskip

\noindent
\emph{Case 1. $u'\in wt$.} In this case we have:
\begin{align*}
 d(t,u')+d(u',w) &=d(t,w)\leq \mathrm{diam}(C)+a=\mathrm{diam}(C)+d(v,u)+d(u,w)\leq\\ &\leq \mathrm{diam}(C)+ d(v,u)+d(u,u')+d(u',w)\leq \\ &\leq \mathrm{diam}(C) +d(v,u)+\delta+d(u',w).
\end{align*}
Hence
$$d(t,u')\leq d(u,v)+\mathrm{diam}(C)+\delta,$$
and it follows that
$$d(t,u)\leq d(t,u')+d(u',u)\leq d(u,v) + \mathrm{diam}(C) +2\delta \leq 8\delta +1 \leq D.$$

\noindent
\emph{Case 2. $u'\in tv$.} Since
$$d(v,u')\geq d(v,u)-d(u,u')\geq 2\delta-\delta=\delta,$$
we have
$$d(t,u')\leq d(t,v) - d(v,u') \leq D - \delta.$$
It follows that
$$d(t,u)\leq d(t,u')+d(u',u)\leq (D-\delta)+\delta=D.$$
In both cases we obtain the desired inequality.
\end{proof}


\end{document}